\documentclass[11pt,a4paper]{article}
\usepackage[english]{babel}
\usepackage[latin1]{inputenc}
\usepackage[T1]{fontenc}
\usepackage{a4}
\usepackage[cyr]{aeguill}
\usepackage{epsfig}
\usepackage{amsmath,amsthm}
\usepackage{dsfont}
\usepackage{amsfonts,amssymb}
\usepackage{float}
\usepackage{fancyhdr}
\usepackage{varioref}
\usepackage{url}
\usepackage{calc}
\usepackage{wrapfig}
\usepackage{graphicx}
\newtheorem{theo}{Theorem}
\labelformat{theo}{Theorem~#1}
\newtheorem{defi}{Definition}[subsection]
\labelformat{defi}{Definition~#1}

\labelformat{remarque}{Remark~#1}
\newtheorem{propo}{Proposition}[subsection]
\labelformat{propo}{Proposition~#1}
\newtheorem{lemme}{Lemma}[subsection]
\labelformat{lemme}{Lemma~#1}
\newtheorem{conjecture}{Conjecture}

\usepackage{authblk}

\usepackage{hyperref}
\hypersetup{
backref=true, %permet d'ajouter des liens dans...
pagebackref=true,%...les bibliographies
hyperindex=true, %ajoute des liens dans les index.
colorlinks=true, %colorise les liens
breaklinks=true, %permet le retour à la ligne dans les liens trop longs
urlcolor= blue, %couleur des hyperliens
linkcolor= blue, %couleur des liens internes
bookmarks=true, %créé des signets pour Acrobat
bookmarksopen=true, %si les signets Acrobat sont créés,
}

\newcommand{\ensemblenombre}[1]{\mathbb{#1}}
\newcommand{\N}{\ensemblenombre{N}}
\newcommand{\E}{\ensemblenombre{E}}
\newcommand{\Z}{\ensemblenombre{Z}}

\newcommand{\R}{\ensemblenombre{R}}

\newcommand{\Scb}{S}
\newcommand{\A}{\mathcal{A}}
\newcommand{\tomega}{\tilde{\omega}}
\newcommand{\Ocb}{\Omega}
\newcommand{\Fcb}{\mathcal{F}}
\newcommand{\Sc}{\tilde{S}}
\newcommand{\Oc}{\tilde{\Omega}}
\newcommand{\Fc}{\tilde{\mathcal{F}}}
\newcommand{\tpi}{\tilde{\pi}}
\newcommand{\1}{\mathds{1}}

\begin{document}

\title{Infinite Volume Continuum Random Cluster Model}

\author[1]{David Dereudre}
\author[1]{Pierre Houdebert}
\affil[1]{Laboratoire de Math\'ematiques Paul Painlev\'e\\University of Lille 1, France}

\maketitle

\begin{abstract}

{The continuum random cluster model is defined as a Gibbs modification of the stationary Boolean model in $\R^d$ with intensity $z>0$ and the law of radii $Q$. The formal unormalized density is given by $q^{N_{cc}}$ where $q>0$ is a fixed parameter and $N_{cc}$ the number of connected components in the random germ-grain structure. In this paper we prove the existence of the model in the infinite volume regime for a large class of parameters including the case $q<1$ or distributions $Q$ without compact support. In the extreme  setting of non integrable radii (i.e. $\int R^d Q(dR)=\infty$) and $q$ is an integer larger than 1, we prove that for $z$ small enough the continuum random cluster model is not unique; two different probability measures solve the DLR equations. We conjecture that the uniqueness is recovered for $z$ large enough which would provide a phase transition result. Heuristic arguments are given. Our main tools are the compactness of level sets of the specific entropy, a fine study of the quasi locality of the Gibbs kernels and a  Fortuin-Kasteleyn representation via Widom-Rowlinson models with random radii. }
  \bigskip


\end{abstract}

\section{Introduction}

In this paper we are interested in a continuum version of the random cluster model usually defined on a deterministic graph. 
The reference model is the stationary Poisson Boolean model with intensity $z>0$ and the law of radii $Q$ a probability measure on $\R^+$. 
It is built by union of balls in $\R^d$ centred to the points of a stationary Poisson point process with intensity $z>0$ and with random independent radii following the distribution $Q$. 
The finite volume continuum random cluster model is then defined as a 
penalized Boolean model
 in some bounded window $\Lambda$. The  unormalized density is given by $q^{N_{cc}}$ where $q>0$ is a positive real number and $N_{cc}$ denotes the number of connected components of the random closed set considered. For this model, the mean number of connected components is increasing with respect to $q$ which provides a clear interpretation of this parameter. For $q=1$ we recover the standard Poisson Boolean model. In the infinite volume regime a global density is senseless and a definition of the continuum random cluster model (called CRCM in the following) via Gibbs modifications is required. 
Precisely a CRCM is a solution of the standard DLR equations \eqref{DLR}. Existence, uniqueness and  non-uniqueness questions arise.

Originally the random cluster model is a lattice model introduced in the late 1960' by Fortuin and Kasteleyn to unify the models of percolation as Ising and Potts models. Most properties and results about this model, such as existence of random cluster model on infinite graphs, percolation property and phase transition property can be found in \cite{ghm,grimbook}. In the continuum setting the CRCM  has been also introduced for its relations with the continuum Potts model and the Widom-Rowlinson model. It led to new proofs of phase transition for those models, see \cite{cck} and \cite{gh}. The CRCM is also studied in stochastic geometry and spatial statistics as an interacting random germ-grain model \cite{A-Moller08}. For a suitable parameter $q$ the CRCM fits as best as possible the clustering of the real dataset. The estimation of the parameter $q$ and the law of radii $Q$ is studied in \cite{A-MollerHel10}.

 All these works, including those in statistical mechanics, involve only the finite volume CRCM. The infinite volume version has not really been studied and highlighted as in its analogous on deterministic graphs. However its interests are numerous in statistical physics and spatial statistics. Involving physical considerations, phase transition phenomenons are observable from infinite volume CRCM; it is believed that the uniqueness of the CRCM would be violated for special critical values of $z$, $q$. Many conjectures and open questions for the lattice models concern the continuum case as well. In stochastic geometry, the infinite volume CRCM provides a more relevant model than the Boolean model for the applications in material science, microemulsion modelling, etc. Its macroscopic properties (mean value, conductivity, permeability) can be studied via stationary tools as Palm theory and ergodic theory. Finally in spatial statistics, the existence of models in infinite volume regime enables the study of the asymptotic properties of estimators, functionals, etc. For example the maximum likelihood estimator of the parameter $q$ along a sequence of increasing observable windows $(\Lambda_n)$ requires the existence of the model in the whole space.

The existence of the infinite volume CRCM has not been proved in a general setting of random radii, continuum parameter $q>0$ and $z>0$. However it is known that it could be constructed via a colour-blind Widom Rowlinson   in the setting where $q$ is an integer values and the radii are not random  \cite{cck}.
 The aim of this paper is to provide the existence of the model for the larger class of parameters as possible. In some case, the  non uniqueness is also proved. Our first theorem gives the existence of the CRCM for any distribution $Q$ with compact support, any $q>0$ and $z>0$. In the case of unbounded radii, the existence is proved if $Q$ has a $d$-moment (i.e. $\int R^dQ(dR)<+\infty$), $q\ge 1$ and $z>0$. In the case where $Q$ does not have a $d$-moment, the existence  is trivial since the Poisson Boolean model  is a CRCM itself. However in a second theorem we prove the existence of another CRCM leading to a non uniqueness Gibbs measures phenomenon. This result is obtained in the case where $q$ is an integer and $z$ is small enough. Using a Pirogov-Sinai approach, this non uniqueness result could provide an interesting tool for proving a phase transition phenomenon in the approximation setting $\int R^dQ(dR)\to \infty$. However we think that it has its own interest as well since  non uniqueness results are quite rare for continuum models. We conjecture that  for $z$ large enough the uniqueness of Gibbs measures is recovered in the non-integrable case. It would provide a phase transition where the uniqueness is lost only for $z$ small enough. This behaviour is unusual for Gibbs point processes where the uniqueness is in general lost for $z$ large enough. Heuristic arguments of the conjecture are given.

The proof of the first theorem is based on the compactness of the level set of the specific entropy and a fine study of the quasi locality of the Gibbs kernels. This strategy have already  successfully applied for proving the existence of several Gibbs models \cite{david,dereudredrouilhetgeorgii, gh}.
In the present paper the very long range dependence is our major problem. Indeed the radii are not bounded and the influence of a ball can be felt far away if it splits a large connected component when it is removed. Such long range dependence were not dealt in the papers mentioned above. In the extreme setting of the second Theorem, we did not succeed to manage the long range dependence of the interaction as in the the first theorem. The non-integrability of the radii may produce balls with too large radii. So we turned to a Fortuin-Kasteleyn representation of the CRCM via a colour-blind Widom-Rowlinson model as in \cite{cck,ghm}. In this setting the DLR equations are simpler to obtain since the non overlapping assumption for balls with different colours confines naturally the range of the interaction. In this setting $q$ represents the number of colors and it is the reason why $q$ is an integer.

Finally note that the standard FKG inequalities, which are abundantly used for the random cluster models on graphs, are not satisfied in the present setting. In particular the thermodynamic limit of finite volume Gibbs measures to the infinite volume Gibbs measure can not be proved. Only the convergence of the empirical field of the finite volume Gibbs measures is obtained.

In Section \ref{section.Notations.Resultats} we introduce the notations and give the formal definition of the CRCM using the DLR formalism. Then we give both main theorems mentioned above in the Section \ref{section.resutats} devoted to the results. The proof of the first existence theorem is given in Section \ref{section.preuve.theo1} and the second theorem in Section \ref{section.preuve.theo2}. The Heuristic arguments of the conjecture are presented in Section \ref{section.conjecture}.

\section{Notations and Results}
\label{section.Notations.Resultats}
\subsection{State space and reference measure}

For $d$ at least 2, $\Scb$ denotes the space $\R^d \times \R^+$ endowed with the Borel $\sigma$-algebra.
$\Ocb$ stands for the set of non negative integer-valued measures $\omega$ on $S$ with finite mass on set $\Lambda\times\R^+$ for any bounded set $\Lambda \subset \R^d$. 
An element $\omega$ of $\Ocb$ is called "configuration" and can be represented as $\omega=\sum_{i\in I} \delta_{(x_i,R_i)}$ for a finite or infinite sequence $(x_i,R_i)_{i\in I}$ of points in $\Scb$ without accumulation points for the sequence $(x_i)_{i\in I}$. 
$\Omega$ is equipped with the classical $\sigma$-algebra $\Fcb$ generated by the counting variables $\omega \mapsto \omega(\Gamma)$ where $\Gamma$ is a bounded Borel subset of $\Scb$. We denote by $\Ocb ^f$ the subset of finite configurations. For a subset $\Lambda$ of $\R^d$, the configuration restricted to $\Lambda$ is defined by $\omega_{\Lambda}(.):=\omega (.\cap \Lambda \times \R^+)$ and $\Fcb _{\Lambda}$ is the sub $\sigma$-algebra of $\Fcb$ generated by the counting variables $\omega \mapsto \omega(\Gamma)$ where $\Gamma$ is a bounded subset of $\Lambda \times \R ^+$. We write $(x,R)\in \omega$ if $\omega(\{(x,R)\})>0$. 
For a configuration $\omega$ and a subset $\Lambda$ of $\R ^d$, $\omega(\Lambda)$ denotes the number of points $(x,R) \in \omega$ such that $x \in \Lambda$.
At each configuration $\omega$ we associate its germ-grain structure 
$$L(\omega)=\underset{(x,R)\in \omega}\bigcup B(x,R)$$ where $B(x,R)$ is the Euclidean closed ball of center $x$ and radius $R$.

For a positive $z>0$ and a probability measure $Q$ on $\R^+$, let $\pi^{z,Q}$ be the distribution on $\Omega$ of the Poisson point process of intensity measure $m= z \lambda ^{(d)} \otimes Q$. It is the distribution of the homogeneous Poisson point process on $\R^d$ with independent marks distributed by $Q$. For $\Lambda \subseteq \R ^d$, $\pi^{z,Q}_{\Lambda}$ denotes the projection of $\pi^{z,Q}$ on $\Lambda \times \R^+$. The random closed set $L$ under the law $\pi^{z,Q}$ is the so-called Poisson Boolean model with intensity $z>0$ and law of radii $Q$.

In the following a probability $P$ on $\Omega$ is called stationary if it is invariant under the translations by vectors in $\Z^d$. A definition with translations by vectors in $\R^d$ could have been considered as well. 

\subsection{Interaction}

 For any configuration $\omega$, the connected components in $L(\omega)$ are defined via the graph of connections $\mathcal{G}(\omega)=(\mathcal{V}(\omega),\mathcal{E}(\omega))$ where the vertices are $\mathcal{V}(\omega)=\{(x,R)\in\omega\}$ and the edges $\mathcal{E}(\omega)=\{\{(x,R),(y,R')\}\subset \mathcal{V}(\omega), \text{ such that }   B(x,R)\cap B(y,R)\neq \emptyset\}$. A connected component in $L(\omega)$ is defined as the union of balls $B(x,R)$ for $(x,R)$ in a connected component of $\mathcal{G}(\omega)$. Note that it could be different from a topological connected component in  $L(\omega)$. For instance the configuration $\omega=\delta_{(0,0)}+\sum_{n=1}^{+\infty}\delta_{(n,n-1/n)}$ has two connected components in $\mathcal{G}(\omega)$ and only one topological connected component in $L(\omega)$. For finite configurations, both definitions are equivalent.

 For $q>0$ fixed, the interaction between the particles is given by the unnormalized density 
$$
q^{ N_{cc}(\omega)}, \ \omega \in \Ocb ^f,
$$
where $N_{cc}(\omega)$ denotes the number of connected components of  $L(\omega)$ (or equivalently in $\mathcal{G}(\omega)$). This density is well defined only for finite configurations. As usual, for infinite configurations we define a local conditional density. 
\begin{propo}
For any $\omega \in \Ocb$ and $\Lambda \subseteq \R^d$ bounded, the following limit
\begin{equation}\label{definitionlimite}
N_{cc}^{\Lambda}(\omega) = \underset{\Delta \to \R^d}{\lim}
\left(  
N_{cc}(\omega_{\Delta}) - N_{cc}(\omega_{\Delta \setminus \Lambda})
\right)
\end{equation}
exists and is called local number of connected components  in $\Lambda$. The limit is taken along any increasing sequence of sets $(\Delta_n)$. 
\end{propo}

\begin{proof}

For a given $\omega \in \Ocb$ we are interested in the quantity 
$c_n = N_{cc}(\omega_{\Delta_n}) - N_{cc}(\omega_{\Delta_n \setminus \Lambda}) $, 
where $(\Delta_n)$ is a given increasing sequence converging to $\R^d$. Since the quantity $c_n$ has integer values, the sequence $(c_n)$ converges if and only if it is constant for $n$ large enough. 
 For a subset $\Lambda\subset \R^d$, a connected component of $L(\omega_{\Lambda^c})$ is called a $\Lambda$-component of $\omega$ if it is connected to $L(\omega_{\Lambda})$.
For any $n\ge 1$, any $X=(x,R)\in \omega_{\Delta_n^c}$ let us introduce the quantity
$$
D_n( X) 
=
N_{cc}(\omega_{\Delta_n} + \delta_X) - N_{cc}(\omega_{\Delta_n \setminus \Lambda} + \delta_X) 
- [ N_{cc}(\omega_{\Delta_n} ) - N_{cc}(\omega_{\Delta_n \setminus \Lambda})  ],
$$
which gives the variation of the number of connected components when the ball $B(x,R)$ is added. It is not difficult to see that $D_n(X)$ may be not zero only if one of the two following situations occurs

\begin{itemize}
\item 
$B(x,R)$ is connected to at least two $\Lambda$-components of $\omega_{\Delta_n}$,

\item
$B(x,R)$ intersects one ball $B$ of $L(\omega_{\Lambda})$ without intersecting any $\Lambda$-component of $\omega_{\Delta_n}$ connected to $B$ (this case happens in particular when $B$ does not intersects any $\Lambda$-component of $\omega_{\Delta_n}$).
\end{itemize}

%\begin{figure}[h]
%\begin{center}
%\includegraphics[width=5cm]{figure100.png}
%\includegraphics[width=5cm]{figure101.png}
%\end{center}
%\end{figure}

Show that there exists $N\ge 1$, which may depend on $\omega$, such that none of the two situations occurs for any $X\in\omega_{\Delta_N^c}$. It ensures that the sequence $(c_n)$ is constant for $n\ge N$. 
Since the number of $\Lambda$-components is finite we can choose $N$ large enough such that the number of $\Lambda$-components in $L(\omega)$ is equal to the number of $\Lambda$-components in $L(\omega_{\Delta_N})$. 
In other words the $\Lambda$-components in $L(\omega)$ are identifiable in $\Delta_N$. 
Now it remains the problem that a ball outside $\Delta_n$ may be connected to $L(\omega_{\Lambda})$ without intersecting any $\Lambda$-component of $\omega_{\Delta_n}$. 
But the number of such balls is finite. So for $N$ large enough this situation does not occur.

\end{proof}

Let us point out that, even if $N_{cc}^{\Lambda}(\omega)$ depends only on $\omega_{\Delta_N}$, the determination of $N$ involves a global knowledge of the configuration $\omega$.
This long range dependence is the major problem in the present paper.

The local number of connected components satisfies the following additivity properties which is a direct consequence of \eqref{definitionlimite}. For any couple of bounded sets $\Lambda \subseteq \Lambda'$ in $\R^d$, there exists a function $\phi_{\Lambda, \Lambda'}$ such that, for all $\omega$ in $\Ocb$

\begin{equation}\label{compatibility}
N_{cc}^{\Lambda'}(\omega)= N_{cc}^{\Lambda}(\omega) + \phi_{\Lambda, \Lambda'}(\omega_{\Lambda^c} ).
\end{equation}

The function $\phi_{\Lambda, \Lambda'}$ depends only on the configurations outside $\Lambda$. It is a crucial point for the compatibility of the Gibbs Kernels. Let us finish this section in giving useful bounds for $N_{cc}^\Lambda$.
\begin{propo}\label{borneimportante}
For any configuration $\omega$ and any bounded set $\Lambda$

\begin{equation}\label{majoimportante}
N_{cc}^{\Lambda}(\omega) \le \omega(\Lambda).
\end{equation}
Moreover, for any $R_0>0$ there exists $K\in\R$ such that for any configuration satisfying for all points $(x,R)\in\omega_ \Lambda$, $R\le R_0$ then

\begin{equation}\label{minoimportante}
N_{cc}^{\Lambda}(\omega)  \ge
K -\omega( \Delta_{R_0} \setminus \Lambda),
\end{equation}
where $\Delta_{R_0}= \Lambda\oplus B(0,R_0+2)$.
 
\end{propo}

\begin{proof}
For any subset $\Delta$ the difference 
$N_{cc}(\omega_{\Delta}) - N_{cc}(\omega_{\Delta \setminus \Lambda} )$ is obviously smaller than $\omega(\Lambda)$ and so its limit when $\Delta$ tends to $\R^d$ as well. The first inequality \eqref{majoimportante} follows.

To get the lower bound for 
$N_{cc}^{\Lambda}(\omega)$, 
we first note that the worst case occurs when $L(\omega_{\Lambda})$ has one connected component which intersects a lot of connected components of $L(\omega_{\Lambda^c})$. So let us control this number of connected components. We consider $(x,R) \in \omega_{\Delta_{R_0}^c}$ such that $B(x,R)$ intersects a connected component of $L(\omega_{\Lambda})$. Since all balls of $L(\omega_{\Lambda})$ have a radius smaller than $R_0$ we have
$$
|B(x,R) \cap \Delta_{R_0} | \geq v_d,
$$
where $v_d$ is the volume of the unit ball in dimension $d$. So the number of connected components of 
$L(\omega_{\Delta_{R_0}^c})$
which are connected to $L(\omega_{\Lambda})$ is bounded from above by $k=\frac{|\Delta_{R_0}|}{v_d}$. 
Taking into consideration the balls in $\omega_{\Delta_{R_0} \setminus \Lambda}$ we have
$$ N_{cc}^{\Lambda}(\omega) 
\geq
1 - k -\omega(\Delta_{R_0} \setminus \Lambda)$$
and \eqref{minoimportante} follows.
\end{proof}

\subsection{Continuum Random Cluster Model}

The continuum random cluster model is defined via standard DLR formalism which requires that the probability measure satisfies equilibrium equations based on Gibbs kernels (see equations \eqref{DLR}). 
Before giving these equations we need to assume that these kernels are well-defined which is the case if for any bounded set $\Lambda$ and any configuration $\omega$ the partition function 
$$ Z_{\Lambda}(\omega_{\Lambda^c})
:=
\int_{\Ocb} q^{N_{cc}^{\Lambda}( \omega'_{\Lambda} + \omega_{\Lambda^c} ) }\pi^{z,Q}_{\Lambda}(d\omega')$$

is non degenerate which means that  $0<Z_{\Lambda}(\omega_{\Lambda^c})<+\infty$. As usual, for any configuration $\omega$, $Z_{\Lambda}(\omega_{\Lambda^c})\ge \pi^{z,Q}_{\Lambda}(0)=e^{-z|\Lambda|}>0$. For the other bound, the following assumption is required

$$\text{ } \qquad \qquad  \qquad q\ge 1  \text{ or the probability measure } Q \text{ has a compact support}. \qquad \qquad  \qquad {(\bf A)}$$
  
\begin{lemme} Under the assumption {\bf (A)}, for any configuration $\omega$ and any bounded set $\Lambda$ the partition function $ Z_{\Lambda}(\omega_{\Lambda^c})$ is finite.
\end{lemme}

\begin{proof}
In the case $q\ge 1$, thanks to \eqref{majoimportante}
\begin{eqnarray*}
Z_{\Lambda}(\omega_{\Lambda^c}) & \le & \int_{\Ocb} q^{ \omega'(\Lambda)}\pi^{z,Q}_{\Lambda}(d\omega')<+\infty.
\end{eqnarray*}
If $q< 1$ and $Q$ has a compact support, there exists $R_0$ such that $Q([0,R_0])=1$ and thanks to \eqref{minoimportante}

\begin{eqnarray*}
Z_{\Lambda}(\omega_{\Lambda^c}) & \le & q^{K-\omega(\Delta_{R_0} \setminus \Lambda)}<+\infty.
\end{eqnarray*}

\end{proof}
 
We are now in position to give the definition of a continuum random cluster model.
\begin{defi}
Under the assumption {\bf(A)}, a probability measure $P$ on $(\Ocb, \Fcb)$ is called a continuum random cluster model for parameters $z$, $Q$ and $q$ (CRCM($z,Q,q$)) if for all bounded $\Lambda \subseteq \R^d$ and all bounded measurable functions $f$ we have 

\begin{equation}\label{DLR}
\int_{\Ocb} f(\omega) P(d\omega)
=
\int_{\Ocb} \int_{\Ocb} f( \omega'_{\Lambda} + \omega_{\Lambda^c} )
\frac{1}{Z_{\Lambda}(\omega_{\Lambda^c})}q^{ N_{cc}^{\Lambda}( \omega'_{\Lambda} + \omega_{\Lambda^c} )}
\pi^{z,Q}_{\Lambda}(d\omega')
P(d\omega).
\end{equation}
Equivalently, for $P$-almost every $\omega$ the conditional law of $P$ given $\omega_{\Lambda^c}$ is absolutely continuous with respect to $\pi^{z,Q}_\Lambda$ with  density $q^{N_{cc}^{\Lambda}(.+\omega_{\Lambda^c}) }/ Z_{\Lambda}(\omega_{\Lambda^c}).$
\end{defi}
These equations, for all $\Lambda$, are called DLR (Dobrushin, Lanford, Ruelle) equations. The existence of such Gibbs measures is the main question of the present paper. The non uniqueness is also considered.

\section{Results} \label{section.resutats}

Our first result theorem ensures the existence of at least one $CRCM(z,Q,q)$ for the larger class of parameters $(z,Q,q)$ as possible.

\begin{theo}\label{bigtheo} ${}$
\begin{itemize}
\item
If $Q$ has a bounded support, i.e there exits $R_0>0$ such that $Q([0,R_0])=1$, then for all $z>0$ and $q>0$ there exists at least one stationary CRCM($z,Q,q$).

\item
If $\int R^d Q(dR)$ is finite, then for all $z>0$ and $q\ge 1$    there exists at least one stationary CRCM($z,Q,q$).

\end{itemize}
\end{theo}

The proof of this theorem is based on the compactness of the level set of the specific entropy (see \ref{compact}). 
This tightness tool allows to build a limit point of a sequence of stationary empirical field coming from the finite volume Gibbs measures. 
Then the main difficulty is to prove that this limit point satisfies the DLR equations. 
This strategy has already been successfully applied for proving the existence of several Gibbs models  \cite{david,dereudredrouilhetgeorgii,gh}. In the present context of continuum random cluster model, the strong non-locality of the interaction is our major problem. 
Indeed the radii are not bounded which produce a long range dependency. 
Moreover the contribution of each ball in the interaction can be long range if the ball is "pivotal" in the sense that it plays a crucial role in the determination of $N_{cc}^{\Lambda}(\omega)$.  The size of the connected components has also an influence on the range of the interaction. In particular, for proving the DLR equations, we need to prove that the limit point has, a priori, at most one infinite connected component. The proof of this theorem is given in Section \ref{section.preuve.theo1}.

In the extreme setting of non-integrable radii  (i.e. $\int R^d Q(dR)= + \infty$). First we note that the existence of a $CRCM(z,Q,q)$ is obvious since the Poisson point process $\pi^{z,Q}$ solves the DLR equations \eqref{DLR}. 

\begin{propo} If $ \int R^d Q(dR)= + \infty$ then the Poisson process $\pi^{z,Q}$ is a $CRCM(z,Q,q)$.

\end{propo}

\begin{proof}

It is well known that, for any bounded set $\Lambda$ and $\pi^{z,Q}$-almost all $\omega$, the set 
$\cup_{(x,R)\in \omega_{\Lambda^c}} B(x,R)$
covers the full space $\R^d$ \cite{chiu2013}. Therefore the function $\omega'_\Lambda \mapsto 
N_{cc}^\Lambda (\omega'_\Lambda + \omega_{\Lambda^c})$ is identically null for $\pi^{z,Q}$-almost every outside configuration $\omega_{\Lambda^c}$. The DLR equations follows easily. 
\end{proof}

Our second theorem ensures the existence of another $CRCM(z,Q,q)$ different from  $\pi^{z,Q}$ when $q$ is an integer and $z$ is small enough. It is a non uniqueness result which proves that the simplex of $CRCM(z,Q,q)$ is not reduced to a singleton.

\begin{theo}\label{bigtheo2}
If $\int_{\R^+} R^d Q(dR)= + \infty$ and if $q$ is an integer larger than 2,
there exists $z_0>0$ such that, for all $z<z_0$, there exists a stationary CRCM($z,Q,q$) different from $\pi^{z,Q}$.
\end{theo}
 
The reason why $q$ must be  an integer comes from the FK representation we used in the proof. Indeed
we are not able to extend the proof of \ref{bigtheo} to the case $\int_{\R^+} R^d Q(dR)= + \infty$. The influence of large balls centred far away is too difficult to control and we do not succeed to prove that the limit point satisfied the DLR equations. Using the representation of the CRCM as a Widom-Rowlinson model (a model of non overlapping balls with $q$ different colors) as in \cite{cck,ghm}, the existence problem becomes simpler. Actually the DLR equations of the Widom-Rowlinson are more "local" since balls with different colors are not allowed to overlap. It produces a natural locality of the interaction. However we think that the assumption $q\in\N$ is only technical and could be relaxed by $q>1$.

Involving the parameter $z$ we believe that the assumption $z$ small enough is crucial. In our proof, it ensures that the CRCM($z,Q,q$) we build is different from $\pi^{z,Q}$. It is based on specific entropy inequalities which ensure the discrimination for $z$ small enough. From a general point of view, we conjecture that for $z$ large enough there exists an unique CRCM($z,Q,q$) which is $\pi^{z,Q}$. The uniqueness would be recovered for $z$ large enough leading to a phase transition phenomenon.

\begin{conjecture}\label{conjecture}
If $\int_{\R^+} R^d Q(dR)= + \infty$, there exists $z_1>0$ such that, for all $z>z_1$, there exists an unique stationary CRCM($z,Q,q$) which is $\pi^{z,Q}$.
\end{conjecture}

Note also that it is unusual in statistical mechanics that the non uniqueness result is obtained for $z$ small (and not large). The proof of the conjecture would reinforce this curious behaviour. Let us finish this section by giving an interpretation of the phase transition conjecture as a competition between the Poisson process and the energy density. Recall that the CRCM on a finite window is a Poisson process with the unormalized density $q^{N{cc}}$. On one hand, since the Poisson process covers completely the space $\R^d$, it influences the CRCM to have an unique connected component which annihilates the energy contribution. The CRCM tends to be a Poisson process and more $z$ is large more this influence is strong. On the other hand the energy density influences the CRCM to have several connected components which tends to seperate the CRCM from the Poisson process. This competition between the Poisson point process and the energy is called Entropy-Energy competition in statistical Physics. We prove in \ref{bigtheo2} that the competition is well balanced for $z$ small enough. Both forces can influence the infinite volume phase. We believe that the Poisson process dominates the competition when $z$ is large enough and it is the sense of the conjecture. Heuristic arguments are given in Section \ref{section.conjecture}.

\section{Proof of \ref{bigtheo} }
\label{section.preuve.theo1}
In Section \ref{section.theo1.bon.candidat} we construct a sequence of finite-volume CRCM $(\bar{P}_n)_n$ from which we extract an accumulation point $\bar{P}$. The compactness (for the local convergence) of level sets of the specific entropy is the main tool here. Then it remains to prove that $\bar{P}$ satisfies the DLR equations. To this end we need to show first that $\bar P$ has at most one unique infinite connected component. This question is addressed in  Section \ref{section.unicité.ccinfini}. Finally in Section \ref{section.DLR.CRCM} the DLR equations are proved. 
The idea is simple, since $\bar P_n$ satisfies the DLR equations and that $(\bar P_n)$ tends to $\bar P$ for the local convergence, we get the DLR equations for $\bar P$ in passing through the limit. However the Gibbs kernels are not local and so a sequence of localizing events has to be introduced.

\subsection{Existence of a limit point} \label{section.theo1.bon.candidat}

For $n$ a positive integer, we set $\Lambda_n=]-n,n]^d$ and we define the finite-volume Gibbs measure with free boundary condition as follow
$$
P_n(d\omega) = P_n^{z,Q,q}(d\omega)
=
\frac{1}{Z_n}q^{N_{cc}(\omega)}\pi^{z,Q}_{\Lambda_n}(d\omega),
$$

where $Z_n=\int_{\Ocb} q^{N_{cc}(\omega)}\pi_{\Lambda_n}^{z,Q,q}(d\omega)$ is the normalizing constant. We need to define a stationary version of $P_n$. Let $\tau_x$ be the translation of vector $x$. 
Then we define $\hat{P}_n=\hat{P}_n^{z,Q,q}$ as the probability measure $\underset{i \in \Z^d}{\otimes} P_n^{z,Q,q} \circ \tau_{2ni}^{-1}$ and finally
$$
\bar{P}_n=\bar{P}_n^{z,Q,q}= \frac{1}{(2n)^d} \underset{i \in I_n}{\sum} \hat{P}_n^{z,Q,q} \circ \tau_i^{-1},
$$

where $I_n=]-n,n]^d \cap \Z^d$. 
Then $\bar{P}_n$ is invariant under the translations $(\tau_i)_{i \in \Z^d}$ (i.e. $\bar{P}_n$ is stationary). Our aim is to find an accumulation point of the sequence $(\bar P_n)$ for the suitable local convergence topology.  

\begin{defi} 
A function $f$ is local if there exists a bounded set $\Lambda$ such that $f(\omega)=f(\omega_{\Lambda})$ for all configurations $\omega$ in $\Ocb$.
A sequence $(\mu_n)$ of measures converges to $\mu$ for the local convergence topology if, for all bounded local functions $f$ we have
$$
\int_{\Ocb} f d\mu_n \underset{n \to \infty}{\longrightarrow} \int_{\Ocb} f d\mu.
$$

\end{defi}

The specific entropy is a powerful tool for proving the tightness for such topology. Let $\mu$ and $\nu$ be two probability measures on $\Ocb$. The relative entropy of $\mu$ with respect to $\nu$ on the set $\Lambda_n$ is defined by

\begin{equation*}
\mathcal{I}_{\Lambda_n}(\mu | \nu)=
\left\lbrace
\begin{array}{ccc}
\int f \ \ln(f) \ d\nu_{\Lambda_n}  & \text{if}& \mu_{\Lambda_n} \ll \nu_{\Lambda_n}, \ f=\frac{d\mu_{\Lambda_n}}{d\nu_{\Lambda_n}},\\
+ \infty & \text{else} &,
\end{array}\right.
\end{equation*}
where $\mu_{\Lambda_n} \ll \nu_{\Lambda_n}$ means that $\mu_{\Lambda_n}$ is absolutely continuous with respect to $\nu_{\Lambda_n}$.

\begin{defi}
Let $\mu$ be a stationary probability measure on $\Ocb$. Then
$$
\mathcal{I}^z(\mu)= \underset{n \to \infty}{\lim} \frac{ \mathcal{I}_{\Lambda_n}(\mu | \pi^{z,Q})}{|\Lambda_n|}
$$
is the specific entropy of $\mu$ with respect to $\pi^{z,Q}$.
\end{defi}

Note that the limit above always exists. We refer to \cite{g} for a general presentation. The following proposition is our tightness tool.

\begin{propo}[Proposition 2.6 \cite{GZ}]\label{compact} For every $c_1,c_2 \geq 0$ the set
$$
\{ \mu \text{ stationary probability measures}, \mathcal{I}^{z}(\mu) \leq c_1i(\mu)+c_2 \},
$$
where $i(\mu)$ is the mean number of points in the box $[0,1]^d$ for the probability measure $\mu$, is compact and sequentially compact for the local convergence topology.

\end{propo}

So by \ref{compact}, to ensure the existence of an accumulation  point for the sequence $(\bar{P}_n)$, we just have to prove an uniform bound for the specific entropy $\mathcal{I}^{z}(\bar{P}_n)$. 

\begin{propo}
For all $n$ we have,
$$
\mathcal{I}^z(\bar{P}_n) \leq z + \max(\ln (q),0) i(\bar{P}_n).$$
\end{propo}

\begin{proof}

First, it is straightforward that by Proposition 15.52 in \cite{g}
  
\begin{align}\label{eg1}
\mathcal{I}^z(\bar{P}_n)=\frac{1}{| \Lambda_n|} \mathcal{I}_{\Lambda_n}(P_n |\pi^{z,Q}),
\end{align}
with
\begin{align}\label{eg2}
\mathcal{I}_{\Lambda_n}(P_n |\pi^{z,Q})
& =
\int_{\Ocb} \ln \left( \frac{q^{N_{cc}(\omega)}}{Z_n} \right) P_n(d\omega)  \nonumber
\\  & = 
-\ln(Z_n) +\ln(q) \int_{\Ocb} N_{cc}(\omega) P_n(d \omega).
\end{align}
Moreover $Z_n \ge
P_n(\omega =0)=
\exp(-z |\Lambda_n| )$ and 

\begin{align}\label{eg4}
0 
\leq
\int_{\Ocb} N_{cc}(\omega) P_n(d \omega)
\leq
\int_{\Ocb} \omega(\R^d) P_n(d \omega)
=
|\Lambda_n| i(P_n)
=
|\Lambda_n| i(\bar{P}_n).
\end{align}
Adding  together \eqref{eg1}, \eqref{eg2} and \eqref{eg4} we get the result.
\end{proof}

The existence of a an accumulation point $\bar{P}=\bar{P}^{z,Q,q}$ follows and for simplicity we write that the sequence $(\bar{P}_n)$ converges to $\bar P$ in place of a subsequence.

For technical reasons involving the DLR($\Lambda$) equation,  the sequence $(\bar{P}_n)$ has to be modified by the the sequence $(\mu_n^{\Lambda})$;

\begin{equation}\label{defimodif}
\mu_n^{\Lambda} 
=
\mu_n^{\Lambda,z,Q,q} 
= 
\frac{1}{(2n)^d} \sum\limits_{\substack{i \in I_n \\ \Lambda \subseteq \tau_i(\Lambda_n)}} P_n^{z,Q,q} \circ \tau_i^{-1}.
\end{equation}

This is no longer a probability measure sequence but the \ref{propo.convlocal.DLR} below shows that the local convergence to $\bar{P}$ holds as well. Moreover each $\mu_n^{\Lambda}$ satisfies the DLR($\Lambda$) equation.
\begin{propo}\label{propo.convlocal.DLR}
For all local bounded functions $f$ we have

$$
\underset{n \to \infty}{lim} \Big{|}\int_M f(\omega)  \ \mu_n^{\Lambda}(d\omega)-\int_M f(\omega) \ \bar{P}_n(d\omega)\big{|}=0 .
$$
and for all $n\ge 1$
$$
\int_{\Ocb} f(\omega) \mu_n^{\Lambda}(d \omega)
=
\int_{\Ocb} \int_{\Ocb} f(\omega'_{\Lambda} + \omega_{\Lambda^c})
\frac1{Z_{\Lambda}(\omega_{\Lambda^c})} q^{N_{cc}^{\Lambda} (\omega'_{\Lambda} + \omega_{\Lambda^c}) }
\pi^{z,Q}_{\Lambda}(d\omega') \mu_n^{\Lambda}(d \omega).
$$
\end{propo}

\begin{proof}
The proof of the first part is given in \cite{david} Lemma 3.5. The proof of the DLR($\Lambda$) equation  for $\mu_n^\Lambda$ is a standard consequence of the compatibility equations \eqref{compatibility}.

\end{proof}

%\begin{proof}
%Let $\Delta$ be the support of $f$. $\Delta$ and $\Lambda$ are bounded, so for $n$ we have large enough $  \Lambda \cup \Delta  \subseteq   \Lambda_n$.
%\begin{align*}
%\delta_n &= |\int_M f(\omega)  \ \mu_n^{\Lambda}(d\omega)-
%\int_M f(\omega) \ \bar{P}_n(d\omega)|
%\\&= 
%\frac{1}{(2n)^d} \left\lvert \sum\limits_{\substack{i \in I_n \\ \Lambda \subseteq \tau_i(\Lambda_n)}} 
%\int_M f(\omega) P_n \circ \tau_i^{-1} (d \omega) - 
%\underset{i \in I_n}{\sum} \int_M f(\omega) \hat{P}_n \circ \tau_i^{-1} (d \omega) \right\rvert
%\\&= 
%\frac{1}{(2n)^d} \left\lvert \sum\limits_{\substack{i \in I_n \\ \Lambda \subseteq \tau_i(\Lambda_n) \\ \Delta \not\subseteq \tau_i(\Lambda_n)}} \int_M f(\omega) P_n \circ \tau_i^{-1} (d \omega) - \sum\limits_{\substack{i \in I_n \\ \Lambda \cup \Delta \not\subseteq \tau_i(\Lambda_n)}} \int_M f(\omega) \hat{P}_n \circ \tau_i^{-1} (d \omega) \right\rvert
%\\ &\leq 
%\frac{1}{(2n)^d} \times 2 \times ||f|| \times \#( \{ i \in I_n, \Lambda \cup \Delta \not\subseteq \tau_i(\Lambda_n) \} ).
%\end{align*}
%Let us take $k \in \mathbb{N}$ such that $\Delta \cup \Lambda \subseteq [-k,k]^d \subseteq \Lambda_n$. 
%We have
%$$
%\#( \{ i \in I_n, \Lambda \cup \Delta \not\subseteq v_i(\Lambda_n) \} ) %\leq 
%2dk \times (2n)^{d-1} ,
%$$
%so 
%$$ \delta_n \leq \frac{4dk}{2n+1} \underset{n \to \infty}{\rightarrow} 0.$$
%\end{proof}

\subsection{Uniqueness of the infinite connected component} \label{section.unicité.ccinfini}

For $k$ in $\N \cup\{\infty\}$, we denote by $\{N^\infty_{cc}= k \}$ (respectively $\{N^\infty_{cc} \leq k \}$) the event of configurations $\omega$ having $k$ (respectively no more than $k$) infinite connected component(s).
This section is devoted to the proof of the following proposition.

\begin{propo}\label{propo.unicité}
Under the assumption $(\bf A)$ we have
$$\bar{P}^{z,Q,q}( \{  N^\infty_{cc} \leq 1 \} ) =1.$$
\end{propo}

%We will prove this result in the usual fashion, following \cite{meeroy}. 
%We will first exclude the case of a finite number, larger than one, of infinite connected components and for the case of a infinite number of infinite connected components we will use the Burton and Keane technic, see \cite{burkea}.

The proof is based on a local modification property which claims that the configurations in a finite box can be modified with positive probability. 

\begin{propo}[Local modification]\label{modiflocal}
Under the assumption $(\bf A)$, for all $\Lambda$ bounded, all $B \in \Fcb _{\Lambda^c}$ satisfying $\bar{P}(B)>0$ and all $A \in \Fcb _{\Lambda}$ satisfying $\pi^{z,Q}(A)>0$ we have
\begin{equation}\label{localmodif}
\bar{P}^{z,Q,q}(A \cap B)>0.
\end{equation}
\end{propo}
\begin{proof}
First for any real number $R_0>0$, let
$
A_{R_0}$ be the event $A \cap \{ \omega \in \Ocb, \forall (x,R) \in \omega_{\Lambda}, \  R \leq R_0 \}.
$
By the monotone convergence Theorem, there is a finite $R_0$ such that $\pi^{z,Q}(A_{R_0}) >0$. 
Since $\bar{P}(A_{R_0} \cap B) \leq \bar{P}(A\cap B)$ it is sufficient to prove the proposition in the special case $A=A_{R_0}$ and that is what we do.
By a martingale theorem, we have 
$\1_B = \underset{\Gamma \to \R^d}{\lim} \E_{\bar{P}}[\1_B | \Fcb_{\Gamma } ]$  $\bar{P}-as$. 
Moreover the function $\E_{\bar{P}}[\1_B | \Fcb_{\Gamma} ]$, that we denote by $\phi_{\Gamma}^B$, is local and the local convergence can be applied. 
\begin{align}
\bar{P}(A \cap B) \label{calculPAB}
&=
\underset{\Gamma \to \R^d}{\lim} 
\int_{\Ocb} \1_A (\omega_{\Lambda}) 
\phi_{\Gamma}^B (\omega_{\Gamma \setminus \Lambda}) \bar{P}(d\omega)  \nonumber
\\ &=
\underset{\Gamma \to \R^d}{\lim} 
\underset{n \to \infty}{\lim}
\int_{\Ocb} \1_A (\omega_{\Lambda}) 
\phi_{\Gamma}^B (\omega_{\Gamma \setminus \Lambda}) \mu_n^{\Lambda}(d\omega) \nonumber
\\ &=
\underset{\Gamma \to \R^d}{\lim} 
\underset{n \to \infty}{\lim}
\int_{\Ocb} \int_{\Ocb} \1_A(\omega'_{\Lambda}) 
\phi_{\Gamma}^B(\omega_{\Gamma \setminus \Lambda})
\frac{ q^{N_{cc}^{\Lambda} (\omega'_{\Lambda} + \omega_{\Lambda^c}) }}
{Z_{\Lambda}(\omega_{\Lambda^c})} 
\pi^{z,Q}_{\Lambda}(d\omega') \mu_n^{\Lambda}(d \omega).
\end{align}
The second and third equalities are obtained by  \ref{propo.convlocal.DLR}. 

From now on we have to separate the cases $q\geq 1$ and $q<1$.

\begin{itemize}
\item Case $q \geq 1$.

From \eqref{majoimportante} we get
\begin{align}
Z_{\Lambda}(\omega_{\Lambda_c}) \label{majoZ}
\leq 
e^{ (q-1)z |\Lambda | }.
\end{align}
From  \eqref{minoimportante}, \eqref{calculPAB} and \eqref{majoZ},  we obtain

\begin{align}
\bar{P}(A \cap B)
& \geq
\underset{\Gamma \to \R^d}{\lim} 
\underset{n \to \infty}{\lim}
\int_{\Ocb} \int_{\Ocb} \1_A(\omega'_{\Lambda}) 
\phi_{\Gamma}^B(\omega_{\Lambda^c})
\frac{q^{ K-\omega(\Delta_{R_0} \setminus \Lambda ) } }
{  e^{(q-1)z|\Lambda |}}
\pi^{z,Q}_{\Lambda}(d\omega') \mu_n^{\Lambda}(d \omega) \nonumber
\\ &=
\underset{\Gamma \to \R^d}{\lim} 
\underset{n \to \infty}{\lim}
\int_{\Ocb} \phi_{\Gamma}^B(\omega_{\Lambda^c})
\frac{q^{ K-\omega(\Delta_{R_0} \setminus \Lambda ) } }
{e^{(q-1)z|\Lambda |}}
\pi_{\Lambda}^{z,Q}(A)\mu_n^{\Lambda}(d \omega)
\nonumber
\\
 \label{inegalitéfin1}
 &=
\frac{q^{K}}{e^{(q-1)z |\Lambda |}}
\pi^{z,Q}_{\Lambda} (A) \int_{\Ocb} \1_B(\omega_{\Lambda^c}) 
q^{-\omega(\Delta_{R_0} \setminus \Lambda)} \bar{P}(d\omega).
\end{align}
which gives $\bar{P}(A \cap B)>0$.

\item Case $q<1$.
From \eqref{minoimportante} and assumption 
$(\bf A)$, which bound the radii in the case $q<1$ , we get
\begin{align}
Z_{\Lambda}(\omega_{\Lambda^c}) \label{majoZ2}
\leq
q^{K-  \omega(\Delta_{R_0} \setminus \Lambda)}  .
\end{align}
From \eqref{majoimportante}, \eqref{calculPAB} and \eqref{majoZ2} we obtain
\begin{align}
\bar{P}(A \cap B) \label{inegalitéfin2}
& \geq
\underset{\Gamma \to \R^d}{\lim} 
\underset{n \to \infty}{\lim}
\int_{\Ocb} \int_{\Ocb} \1_A(\omega'_{\Lambda}) 
\phi_{\Gamma}^B(\omega_{\Lambda^c})
\frac{q^{ \omega'(\Lambda )} }
{q^{K - \omega(\Delta_{R_0} \setminus \Lambda   )} }
\pi^{z,Q}_{\Lambda}(d\omega') \mu_n^{\Lambda}(d \omega) \nonumber
\\ &=
\underset{\Gamma \to \R^d}{\lim} 
\underset{n \to \infty}{\lim}
\int_{\Ocb} \phi_{\Gamma}^B(\omega_{\Lambda^c})
\frac{q^{ \omega(\Delta_{R_0} \setminus \Lambda) } }
{q^{K} }
\mu_n^{\Lambda}(d \omega)
\int_{\Ocb} \1_A \
q^{ \omega'(\Lambda) }
d \pi^{z,Q}_{\Lambda}  \nonumber
\\ &=
\int_{\Ocb} \1_B(\omega_{\Lambda^c})
\frac{q^{ \omega(\Delta_{R_0} \setminus \Lambda) } }
{q^{(1- K )} }
\bar{P}(d \omega)
\int_{\Ocb} \1_A (\omega') \ q^{ \omega'(\Lambda) }
\pi^{z,Q}_{\Lambda}(d\omega').
\end{align}
which gives $\bar{P}(A \cap B)>0$ as well.
\end{itemize}

\end{proof}

Using the \ref{modiflocal}, we are now in position to prove \ref{propo.unicité} in following a standard strategy in percolation theory. We just give a sketch of the proof and we refer to \cite{meeroy} for details. First we represent $\bar P$ as a mixture of extremal ergodic stationary probability measures $\bar P=\int \bar P_\theta \Theta(d\theta)$ where each $\bar P_\theta$ satisfies the local modification \eqref{localmodif} property. We show now that for $\Theta$-a.s. all $\theta$, $\bar P_\theta(N_{cc}^\infty\le1)=1$. By ergodicity of $\bar P_\theta$, the number of infinite connected components is $\bar P_\theta$-almost surely constant. The case of a finite number, larger than one, infinite connected components is excluded thanks to the local modification property. The case of an infinite number of infinite connected components is also excluded by a Burton and Keane argument \cite{burkea}.

In the next section, the $d$-moment assumption (i.e. $\int R^d Q(dr)<\infty$) appears for the first time in the proof of \ref{bigtheo}. In particular it is not required in the proof of \ref{propo.unicité} above which will be usefull in the proof of \ref{bigtheo2} in Section \ref{section.preuve.theo2}. 

\subsection{DLR equations} \label{section.DLR.CRCM}
In this section, we fix the bounded set $\Lambda$ and we show the DLR($\Lambda$) equation. To this end  sequences $(W_{i,j})$ and $(A_{i,j})$ of events are defined on which the variable $N_{cc}^\Lambda$ is local and such that the probabilities $\bar P(W_{i,j})$ and $\bar P(A_{i,j})$ tend to one when $i$ and $j$ tend to infinity in a good way. Without loss of generality  we assume that the function $f$ in the DLR($\Lambda$) equation is local and satisfies, for a finite $R_0$,  $f(\omega)=0$ as soon as there is $(x,R)$ in $\omega_{\Lambda}$ with $R>R_0$.
The general case is obtained by standard approximations.

\begin{defi}
Let $\Delta_i=[-i,i]^d$. For $j>i$ we define
\begin{itemize}
\item
$ A_{i,j}
=
\{
\omega \in \Ocb, \ \forall (x,R) \in \omega_{\Delta_j^c},
B(x,R) \cap \Delta_i = \emptyset 
\} $,
\item
$ W_{i,j}$ the event of
$\omega$ in $\Ocb$ having at most one connected component of 
$L(\omega_{\Delta_j \setminus \Lambda })$ which intersects $\Lambda_{R_0}$ and $\Delta_i^c$,
where $\Lambda_{R_0}$ is the set $\Lambda \oplus B(0,R_0)$.
\end{itemize}
\end{defi}
Before investigating  the probability of those events, the next lemma shows the "localization" of the functional $N_{cc}^\Lambda$.

\begin{propo}\label{propo.localisation.H}
For all $j>i$ large enough (depending on $\Lambda$ and $R_0$) and for all $\omega$ in
$A_{i,j} \cap W_{i,j}$
\begin{align}
N_{cc}^\Lambda(\omega)
=
N_{cc}^\Lambda(\omega_{\Delta_j}). \nonumber
\end{align}
\end{propo}

\begin{proof}
Since $\omega$ is in $W_{i,j} \cap A_{i,j}$, each balls of $L(\omega_{\Delta_j^c})$ does not hit $L(\omega_{\Lambda})$ and does not hit two or more $\Lambda$-components of $L(\omega_{\Delta_i})$. Therefore
$$
N_{cc}^\Lambda(\omega) 
= 
N_{cc}(\omega_{\Delta_j}) - N_{cc}(\omega_{\Delta_j \setminus \Lambda})
=
N_{cc}^\Lambda(\omega_{\Delta_j}).
$$
\end{proof}

The probability of the events $(W_{i,j})$ and $(A_{i,j})$ now have to be controlled. Involving the events $W_{i,j}$  we have the following proposition.

\begin{propo}\label{propo.proba.Wij}
\begin{equation} \label{limite}
\underset{i \to \infty}{\lim}\underset{j \to \infty}{\lim}  \bar{P}(W_{i,j})=1.
\end{equation}
\end{propo}

\begin{proof}
We have that
%$$
%\underset{j>i}{\bigcap} W_{i,j}=\{  \omega, L(\omega_{\Lambda^c}) 
%\text{ has at most one cc intersecting } \Lambda_{R_0} 
%\text{ and } \Delta_i^c  \},
%$$
%so
$$
\underset{i \in \N}{\bigcup} 
\underset{j>i}{\bigcap} W_{i,j}=\{  \omega, L(\omega_{\Lambda^c}) 
\text{ has at most one infinite cc intersecting } \Lambda_{R_0}   \}.
$$
Then 
\begin{align} \label{eq.preuve.Wij}
\underset{i \to \infty}{\lim}\underset{j \to \infty}{\lim}  \bar{P}(W_{i,j})
&=
\bar{P}\left( \underset{i \in \N}{\bigcup} 
\underset{j>i}{\bigcap} W_{i,j} \right)
\nonumber
\\ & \geq
\bar{P} ( L(\omega_{\Lambda^c}) \text{ has at most one infinite cc } )
=1.
\end{align}
To prove the last equality in \eqref{eq.preuve.Wij}, let suppose that with positive probability $L(\omega_{\Lambda^c})$ has at least two infinite connected components, then using the local modification result (\ref{modiflocal}),
\begin{align}
\bar{P}(L(\omega) \text{ has at least two infinite connected components})>0,
\nonumber
\end{align}
 which is a direct contradiction of \ref{propo.unicité}.
 
\end{proof}

The control of the probability of $A_{i,j}$ is a bit harder to obtain. Since $A_{i,j}$ is not a local event, the probabilities 
$\mu_n^{\Lambda}(A_{i,j})$ need to be controlled uniformly on $n$.

\begin{propo}\label{propo.proba.Aij} Under the assumptions of \ref{bigtheo}, meaning bounded radii or $q\ge 1$ and $\int_{\R^+} R^d Q(dR)< + \infty$, then for all $i\ge 1$ 
$$
 \underset{j \to \infty}{\lim} \max (\bar{P}(A^c_{i,j}),\sup_n \mu_n^{\Lambda}(A^c_{i,j})) =0.
 $$
\end{propo}
 \begin{proof}

The case of bounded radii is quite simple since for $j>i + R_0$, $\bar{P}(A^c_{i,j})=0$ and $\mu_n^{\Lambda}(A^c_{i,j})=0$. In the case $q \geq 1$ and $\int R^d Q(dR) < + \infty$, we use  stochastic comparison results in \cite{gk} to compare $\mu_n^{\Lambda}$ with respect to $\pi$. Recall standard definitions on stochastic domination for point processes. An event $A$ is called increasing if for any $X=(x,R)$ and any configuration $\omega$, then $\omega + \delta_X\in A$ as soon as $\omega\in A$. If $\mu$ and $\nu$ are two probability measures on $\Ocb$, we say that $\nu$ dominates $\mu$ if $\mu(A) \leq \nu(A)$ for all increasing set $A$.

For any $(x,R)$ and any finite configuration $\omega$ the difference $N_{cc}(\omega+\delta_{(x,R)})-N_{cc}(\omega)$ is at most one. Therefore, thanks to Theorem 1.1 in \cite{gk}, $P_n$ is stochastically dominated by $\pi^{zq,Q}$, and for any increasing event $A$ we have $\mu_n^\Lambda (A) \le \pi^{qz,Q}(A)$.
Since the event  $A_{i,j}^c$ is increasing, we have the inequality 
 
 \begin{equation}\label{eq1}
 \mu_n^\Lambda (A^c_{i,j}) \le \pi^{qz,Q}(A^c_{i,j}).
 \end{equation}
 
In considering the events
$$ A_{i,j,k}
=
\{
\omega \in \Ocb, \ \forall (x,R) \in \omega_{\Delta_j^c\cap \Delta_k},
B(x,R) \cap \Delta_i = \emptyset 
\} $$
we have 

\begin{equation}\label{eq2}
 \bar P(A_{i,j}^c)=\lim_k \bar{P}(A_{i,j,k}^c)= \lim_k \lim_n \mu_n^\Lambda (A_{i,j,k}^c) \le \lim_k \lim_n \mu_n^\Lambda (A_{i,j}^c)\le \pi^{qz,Q}(A^c_{i,j}).
 \end{equation}

It is well-known that the number of balls in a Poisson boolean model (with intensity measure $m=z\lambda^{(d)}\otimes Q$) which intersects a bounded set $\Delta$ is a Poisson random variable with parameter $z\int( \lambda^{(d)} (\Delta\oplus B(0,R))Q(dR)$ (See \cite{chiu2013} for instance). Since $\int_{\R^+} R^d Q(dR) < + \infty$, this parameter is finite and the random variable is almost surely finite. We deduce that $ \underset{ j \to \infty}{\lim} \pi^{zq,Q}(A_{i,j}^c) = 0$ and the \ref{propo.proba.Aij} follows from \eqref{eq1} and \eqref{eq2}.

 \end{proof}

We are in position to prove the DLR($\Lambda$) equation. Consider the quantity $$
\delta 
= 
\Big{|} \int_{\Ocb} f d \bar{P} - \int_{\Ocb} \int_{\Ocb}
f(\omega'_{\Lambda} + \omega_{\Lambda^c} ) 
\frac 1{Z_{\Lambda}(\omega_{\Lambda^c})} q^{N_{cc}^\Lambda( \omega'_{\Lambda} + \omega_{\Lambda^c})}
\pi_{\Lambda}(d\omega') \bar{P}(d\omega)
\Big{|}$$
where $f$ is also assumed to be bounded by $1$. Let us show that $\delta$ is arbitrary smaller than any $\epsilon>0$. 

By \ref{propo.proba.Wij} and \ref{propo.proba.Aij} we choose $i<j$ large enough such that $\bar{P}(A_{i,j}^c \cup W_{i,j}^c ) \leq \epsilon$ and $\mu_n^{\Lambda}(A_{i,j}^c) \leq \epsilon$ for all $n$. So

\begin{align*}
\delta \leq
\Big{|} \int_{\Ocb} f d \bar{P} - \int_{\Ocb} \int_{\Ocb} 
\1_{A_{i,j} \cap W_{i,j}}(\omega_{\Lambda^c})
f(\omega'_{\Lambda} + \omega_{\Lambda^c} ) 
\frac {q^{N_{cc}^{\Lambda}( \omega'_{\Lambda} + \omega_{\Lambda^c})}}{Z_{\Lambda}(\omega_{\Lambda^c})} 
\pi_{\Lambda}(d\omega') \bar{P}(d\omega)
\Big{|} + \epsilon,
\end{align*}
thanks to \ref{propo.localisation.H}
\begin{align*}
\delta & \leq
\Big{|} \int_{\Ocb} f d \bar{P} - \int_{\Ocb} \int_{\Ocb} 
\1_{A_{i,j} \cap W_{i,j}}(\omega_{\Lambda^c})
f(\omega'_{\Lambda} + \omega_{\Lambda^c} ) 
\frac  {q^{N_{cc}^{\Lambda}( \omega'_{\Lambda} + \omega_{\Delta_j \setminus \Lambda}  ) }}{Z_{\Lambda}(\omega_{\Delta_j \setminus \Lambda})}
\pi_{\Lambda}(d\omega') \bar{P}(d\omega)
\Big{|} + \epsilon
\\ &\leq
\Big{|} \int_{\Ocb} f d \bar{P} - \int_{\Ocb} \int_{\Ocb} 
\1_{W_{i,j}}(\omega_{\Lambda^c})
f(\omega'_{\Lambda} + \omega_{\Lambda^c} ) 
\frac{q^{N_{cc}^{\Lambda}( \omega'_{\Lambda} + \omega_{\Delta_j \setminus \Lambda}  ) }}
{Z_{\Lambda}(\omega_{\Delta_j \setminus \Lambda})}
\pi_{\Lambda}(d\omega') \bar{P}(d\omega)
\Big{|} + 2 \epsilon.
\end{align*}
But since $\1_{W_{i,j}}$ is a local function, we can use the local convergence of $\mu_n^{\Lambda}$ to $\bar{P}$. So for a $n$ large enough (which depends on $i$ and $j$ fixed) we have
\begin{align*}
\delta  
&\leq
\Big{|} \int_{\Ocb} f d \mu_n^{\Lambda} - \int_{\Ocb} \int_{\Ocb} 
\1_{W_{i,j}}(\omega_{\Lambda^c})
f(\omega'_{\Lambda} + \omega_{\Lambda^c} ) 
\frac{q^{N_{cc}^{\Lambda}( \omega'_{\Lambda} + \omega_{\Delta_j \setminus \Lambda}  ) }}
{Z_{\Lambda}(\omega_{\Delta_j \setminus \Lambda})}
\pi_{\Lambda}(d\omega') \mu_n^{\Lambda}(d\omega)
\Big{|} + 3 \epsilon
\\ & \leq
\Big{|} \int_{\Ocb} f d \mu_n^{\Lambda} - \int_{\Ocb} \int_{\Ocb} 
f(\omega'_{\Lambda} + \omega_{\Lambda^c} ) 
\1_{A_{i,j} \cap W_{i,j}} (\omega_{\Lambda^c})
\frac{q^{N_{cc}^{\Lambda}( \omega'_{\Lambda} + \omega_{ \Lambda^c}  ) }}
{Z_{\Lambda}(\omega_{\Lambda^c})}
\pi_{\Lambda}(d\omega') \mu_n^{\Lambda}(d\omega)
\Big{|} + 4 \epsilon
\\ & \leq
\Big{|} \int_{\Ocb} f d \mu_n^{\Lambda} - \int_{\Ocb} \int_{\Ocb} 
\1_{W_{i,j}}
f(\omega'_{\Lambda} + \omega_{\Lambda^c} ) 
\frac{q^{N_{cc}^{\Lambda}( \omega'_{\Lambda} + \omega_{ \Lambda^c}  ) }}
{Z_{\Lambda}(\omega_{\Lambda^c})}
\pi_{\Lambda}(d\omega') \mu_n^{\Lambda}(d\omega)
\Big{|} + 5 \epsilon.
\end{align*}
By \ref{propo.convlocal.DLR} $\mu_n^{\Lambda}$ satisfies the  DLR($\Lambda$) equation and so we get
\begin{align*}
\delta
& \leq 
5 \epsilon + \mu_n^{\Lambda}(W_{i,j}^c)
\leq 6 \epsilon.
\end{align*}

\section{Proof of \ref{bigtheo2}}\label{section.preuve.theo2}
In the proof of \ref{bigtheo}, the assumptions of bounded radii or integrable radii are important to localize the local conditional densities, using comparison tools. This proof can not be adapted to the case of non integrable radii and we turn to  another strategy based on a Fortuin-Kasteleyn representation of the CRCM via  a Widom-Rowlinson model. 
In Section \ref{section.def.WR} the Widom-Rowlinson model is defined as a random balls model with $q$ different colours such that balls with different colours are not allowed to overlap. In Section \ref{section.preuve.passage.WR.CRCM} we show that a colour blind Widom-Rowlinson model (i.e. the colours are forgotten) is a CRCM. The uniqueness of the infinite connected component is required in this identification. In the last Section \ref{section.existence.WR} we show the existence of a Widom-Rowlinson model having this uniqueness property and such that for $z$ small enough it has almost surely at least two different colours.  This ensures that the associated CRCM is different to the Poisson process and the \ref{bigtheo2} is proved.

\subsection{Widom-Rowlinson model}
\label{section.def.WR}
 Starting now $q$ is an integer larger than 1 and is the number of colours in the model. 
 Let $\Sc$ denote the new state space $\R^d \times \R^+ \times \{1, \dots ,q \}$. $\Oc$ is the set of coloured configurations, embedded with the classical $\sigma$-algebra $\Fc$. To avoid confusions we write $\tilde \omega$ for a coloured configuration and $\tpi=\tpi^{z,Q,q}$ denotes the law of a Poisson point process on $\Oc$ with intensity measure 
 $\tilde{m}= z \lambda^{(d)} \otimes Q \otimes \mathcal{U}_q$ where $\mathcal{U}_q$ stands for the uniform law on the set $\{1, \dots , q \}$. 
 As before, for a subset $\Lambda$ of $\R ^d$, $\tilde{\omega}_{\Lambda}$, $\Fc_{\Lambda}$ and $\tpi_{\Lambda}$ are the restrictions on the set $\Lambda \times \R^+ \times \{1, \dots , q\}$ of the respective  objects. 
The set of authorized configurations $\A$ is defined as followed.
$$
\A = \{ 
\tilde{\omega} \in \Oc, \forall (x,R,k),(x',R',k') \in \tilde{\omega}, \ k \not = k' \Rightarrow |x-x'| > R+R'
\}.
$$
In this set two balls of different colours do not overlap.
\begin{defi}\label{defi.WR}
A probability measure $\mu$ on $\Oc$ is a Widom-Rowlinson model for parameters $z$, $Q$ and $q$ (WR($z,Q,q$)) if it satisfies the two following properties.
\begin{itemize}
\item
$\mu(\A)=1$.
\item DLR equations :
For all bounded set $\Lambda$, for all bounded functions $f$
$$
\int_{\Oc} f \ d \mu
=
\int_{\Oc} \int_{\Oc}
f( \tilde{\omega}'_{\Lambda} + \tilde{\omega}_{\Lambda^c}  )
\frac{\1_{\A} ( \tilde{\omega}'_{\Lambda} + \tilde{\omega}_{\Lambda^c}  )}
{ \tilde{Z}_{\Lambda}(\tilde{\omega}_{\Lambda^c})   }
\tpi_{\Lambda}(d\tomega') \mu( d \tomega ),
$$
where 
$ \tilde{Z}_{\Lambda}(\tilde{\omega}_{\Lambda^c})
=
\int_{\Oc} 
\1_{\A} ( \tilde{\omega}'_{\Lambda} + \tilde{\omega}_{\Lambda^c}  )
\tpi_{\Lambda}(d\tomega')
$.
\end{itemize}
\end{defi}

In the case of deterministic radii, the Widom-Rowlinson model have been first introduced in \cite{Widom70} and studied in \cite{cck,ghm} for example. Those papers are mainly devoted to the case of deterministic radii but can be easily extended to the case of bounded random radii. Here we investigate the more complicated non integrable case $\int_{\R^+} R^d Q(dR) = + \infty$. 
Note that the Poisson point process $\pi^{z/q,Q}$ coloured by a single colour is a WR($z,Q,q$).
In the following we are interested in a mixed phase which ensures almost surely the simultaneous existence of at least two colours.

\subsection{The Fortuin-Kasteleyn representation}
\label{section.preuve.passage.WR.CRCM}

First let us introduce the notions of colour-blind measure and colouration kernel.
\begin{defi}
From a probability measure $\mu$ on $(\Oc, \Fc)$ we define the two following quantities :
\begin{itemize}
\item The colour-blind probability measure $\mu_{cb}$ on $(\Ocb, \Fcb)$ by $\mu_{cb}(B)= \mu (B \times \{1, \dots ,q \} )$ for each $B\in\Fcb$.

\item The colouration kernel $C_{\mu}$ on $\Fc \times \Ocb$ defined by  $C_{\mu}(A |.)= \E_{\mu} [\1_A | \Fcb  ] $, where $\Fcb$ is consider here as a sub $\sigma$-algebra of $\Fc$.
\end{itemize}
\end{defi}

In other words, $\mu_{cb}$ is the law of the random balls model $\mu$ where the colours are forgotten and $C_\mu$ is the
kernel which gives the distributions of colours for the random balls model $\mu$ given the configuration of balls. The relation between the CRCM and the Widom-Rowlinson model is given in the following proposition.

\begin{propo}\label{propo:passage:wr:crcm}
If $\mu$ is a WR($z,Q,q$) with at most one infinite connected component, then the colour-blind measure $\mu_{cb}$ is a CRCM($z/q,Q,q$).
\end{propo}

This result is well know in finite volume and has been used to give phase transition result in \cite{cck} for the Widom-Rowlinson model and in \cite{gh} for the larger class of continuum Potts models. Note also that the standard Fortuin-Kasteleyn representation gives a colouration procedure on the CRCM($z/q,Q,q$) in order to recover a  WR($z,Q,q$). We ommit this part here. Note also that no assumption in \ref{propo:passage:wr:crcm} is required on $Q$ or $z$.

\begin{proof} [Proof of \ref{propo:passage:wr:crcm}]

One can carry out the proof  using the DLR equations, but it turns out it is slightly easier using the equivalent definition of Gibbs measures via the GNZ equations \cite{NZ}.

\begin{lemme}[GNZ equation]\label{lemme.gnz}
A probability measure $\mu$ on $\Oc$ is a WR($z,Q,q$) if and only if, for all measurable bounded $F$ 
\begin{align}
\int_{\Oc} \underset{\tilde{X} \in \tomega}{\sum} F(\tomega - \delta_{\tilde{X}}, \tilde{X}) \mu(d\tomega)
=
\int_{\Oc} \int_{\Sc} F(\tomega, \tilde{X}) \1_{\A}( \tomega + \delta_{\tilde{X}}) \tilde{m}(d\tilde{X}) \mu(d \tomega),  \label{eq.GNZ.WR}
\end{align}
where $\tilde{X}= (x,R,k)$ and $\tilde{m}= z\lambda^d \otimes Q \otimes \mathcal{U}_q$.

A probability measure $\nu$ on $\Ocb$ is a CRCM($z,Q,q$) if and only if, for all measurable bounded $F$ 
\begin{align}
\int_{\Ocb} \underset{X \in \omega}{\sum} F(\omega - \delta_X, X) \nu(d\omega)
=
\int_{\Ocb} \int_{\Scb} F(\omega, X) q^{N_{cc}^{\Lambda}(\omega + \delta_X)-N_{cc}^{\Lambda} (\omega) } m(dX) \nu(d \omega), \label{eq.GNZ.CRCM}
\end{align}
where $X= (x,R)$, $m= z\lambda^d \otimes Q $ and where $\Lambda$ is any bounded subset of $\R ^d$ containing $x$.
\end{lemme}

We need a standard result which describes the colouration of the finite connected components in a WR($z,Q,q$).

\begin{lemme}\label{lemme:couleur:WR:ccfini}
Let $\mu$ be a WR($z,Q,q$). For $\mu$ almost all $\tilde \omega$ the colour of a given finite connected component in $C_\mu(.|\omega)$  is independent of the colours of all other finite or infinite connected components, and its law is uniform on $\{1,\ldots,q\}$.
\end{lemme}
This lemma is a straightforward consequence of the DLR equations satisfied by $\mu$. Details are omitted.

Now let us prove that $\mu_{cb}$ satisfies \eqref{eq.GNZ.CRCM} given that $\mu$ satisfies \eqref{eq.GNZ.WR}. Let $F$ be a bounded measurable function from $\Ocb \times \Scb$  to $\R$ and its associated extension $\tilde{F}$ on $\tilde\Ocb \times \tilde \Scb$ defined by $\tilde{F}(\tomega, \tilde{X})=F(\omega, X)$, where $\tilde{X}=(x,R,k)$, $X=(x,R)$ and $\omega$ is the projection of $\tomega$ on $\Ocb$ the space of not-coloured configurations. 
\begin{align}
\int_{\Ocb} & \underset{X \in \omega}{\sum}  F(\omega - \delta_X, X) \mu_{cb}(d\omega)
\nonumber
=
\int_{\Oc} \underset{\tilde{X} \in \tomega}{\sum} 
\tilde{F}(\tomega - \delta_{\tilde{X}}, \tilde{X}) \mu(d\tomega)
\nonumber
\\ &=
\int_{\Oc} \int_{\Sc} \tilde{F}(\tomega, \tilde{X}) 
\1_{\A}( \tomega + \delta_{\tilde{X}}) \tilde{m}(d\tilde{X}) \mu(d \tomega)
\nonumber
\\ &=
\int_{\Ocb} \int_{\Scb} F(\omega, X) 
\int_{\Oc} \underset{k \in \{1,..q \}}{\sum} \frac{1}{q}
\1_{\A}( \tomega + \delta_{(X,k)})
C_{\mu}(d \tomega | \omega )
 m(dX) \mu_{cb}(d \omega),
 \label{eq.calcul.gnz.1}
\end{align}
where $\tilde{X}=(X,k)$ is a coloured point. 
The indicator function  $\1_{\A}( \tomega + \delta_{(X,k)})$ in \eqref{eq.calcul.gnz.1} is equal to one if and only if all connected components of $L(\tomega)$ hitting the ball $B(x,R)$ have the same colour $k$. Now let us consider that $\omega$ and $X=(x,R)$ are fixed and we denote by $j$ the number of connected components in $L(\omega)$ hitting  the ball $B(x,R)$. Since the number of infinite connected components is at most one, two cases are possible :
\begin{itemize}
\item All of those $j$ connected components are finite. Thanks to \ref{lemme:couleur:WR:ccfini} 
 
\begin{align}
\int_{\Oc} \underset{k \in \{1,..q \}}{\sum} \frac{1}{q}
\1_{\A}( \tomega + \delta_{(X,k)})
C_{\mu}(d \tomega | \omega )
&=
\underset{k \in \{1,..q \}}{\sum} \frac{1}{q}   \int_{\Oc} 
\1_{\A}( \tomega + \delta_{(X,k)})
C_{\mu}(d \tomega | \omega )
\nonumber
\\ &= 
\underset{k \in \{1,..q \}}{\sum} \frac{1}{q} \frac{1}{q^j}
= \frac{1}{q^j}.
\label{eq.colo.gnz}
\end{align}
\item One of those connected components is infinite. Thanks to \ref{lemme:couleur:WR:ccfini}, the colours of finite connected components are independent of each other and independent of the colour of the infinite connected component. Therefore similar computations give the same result as in \eqref{eq.colo.gnz}.
\end{itemize}

Adding \eqref{eq.colo.gnz} into \eqref{eq.calcul.gnz.1} we obtain

\begin{align}
\int_{\Ocb} \underset{X \in \omega}{\sum}  F(\omega - \delta_X, X) \mu_{cb}(d\omega)
&=
\int_{\Oc} \int_{\Sc} F(\omega, X) \frac{1}{q^j} m(dX) \mu_{cb}(d\omega)
\nonumber
\\ &=
\int_{\Oc} \int_{\Sc} F(\omega, X)q^{1-j} \frac{1}{q} m(dX) \mu_{cb}(d\omega)
\nonumber
\\ &=
\int_{\Oc} \int_{\Sc} F(\omega, X)
q^{N_{cc}^{\Lambda}(\omega + \delta_X) -N_{cc}^{\Lambda}(\omega) }
\frac{1}{q} m(dx) \mu_{cb}(d\omega)
\label{eq.calcul.gnz.final}
\end{align}
which is exactly the  GNZ equation for a CRCM($z/q,Q,q$). The proposition is proved.

\end{proof}

\subsection{Existence of a Widom-Rowlinson model} 
\label{section.existence.WR}

In order to use \ref{propo:passage:wr:crcm} for proving \ref{bigtheo2}, we need the following existence result.

\begin{propo}\label{existence.WR}
If $\int_{\R^+} R^d Q(dR) = + \infty$, there is a critical $z_c >0$ such that, for all $z<z_c$, there exists a stationary $WR(z,Q,q)$ having at most one infinite connected component. Moreover, with probability one there is at least two balls with different colours.
\end{propo}

The proof of this result follows the same scheme as the proof of \ref{bigtheo}. 
First in Section \ref{section.WR.construction.candidat} we construct a limit point via a finite volume approximation sequence. The uniqueness of the infinite connected component is proved as in Section \ref{section.unicité.ccinfini}. Our strategy for proving the DLR equations is based on a sequence of shield events presented in Section \ref{section.WR.shield}.  They are related to boxes containing balls with different colours and therefore in Section \ref{section.WR.preuve.propo.2couleurs} we  prove that the limit point do not produce almost surely an unique colour. This is done by comparing its specific entropy to the class of monochromatic probability measures. At this point the assumption "small $z$" seems crucial. Finally in Section \ref{section.WR.DLR} the DLR equations are proved.

\subsubsection{Existence of a limit point}
\label{section.WR.construction.candidat}

As in the previous section, the finite volume Widom-Rowlinson measure with free boundary condition is defined as 
$$
\nu_n (d \tomega)=\nu_n^{z,Q,q} (d \tomega)= \frac{ \1_{\A}(\tomega) }{\tilde{Z}_n} \tpi_{\Lambda_n}^{z,Q,q}(d\tomega).
$$
The probability measures $\hat{\nu}_n= \underset{i \in \Z ^d}{\otimes} \nu_n \circ \tau_{2ni}^{-1}$ and 
$\bar{\nu}_n=\frac{1}{\# I_n} \underset{i \in I_n}{\sum} \hat{\nu}_n \circ \tau_i^{-1}$ are defined from $\nu_n$ as in Section \ref{section.WR.construction.candidat}. Moreover the local convergence topology, the specific entropy  and the tighness tools are similar and lead to the following result. 

\begin{propo}\label{propo.entro.nu n}
For all $n$
$$
\mathcal{I}^z (\bar{\nu}_n^{z,Q,q})
\leq z,
$$
which ensures the existence of an accumulation point $\bar{\nu}^{z,Q,q}=\bar{\nu}$ of the sequence $(\bar{\nu}_n)$ for the local convergence topology. 
\end{propo}
\begin{proof}
By Proposition 15.52 in \cite{g}
\begin{align}
\mathcal{I}^z (\bar{\nu}_n^{z,Q,q})
&=
\frac{1}{| \Lambda_n |} \mathcal{I}_{\Lambda_n} (\nu_n^{z,Q,q} | \tpi^{z,Q,q}) \nonumber
\\ &=
\frac{1}{| \Lambda_n |} \int_{\Oc} \ln \left( \frac{ \1_{\A}(\tomega)}{Z_n}   \right) \frac{ \1_{\A}(\tomega)}{Z_n} 
\tpi_{\Lambda_n}^{z,Q,q}(d \tomega) \nonumber
\\
\mathcal{I}^z (\bar{\nu}_n^{z,Q,q})
 &=
\frac{- \ln (Z_n)}{| \Lambda_n |} \leq z . \nonumber
\end{align}
\end{proof}

For simplicity we suppose that $(\bar{\nu}_n)$ converges to $\bar{\nu}^{z,Q,q}=\bar{\nu}$ (without taking a subsequence). Let us start to investigate the property of $\bar \nu$.

\begin{propo}\label{propo.proba.A}
$
\bar{\nu}(\A)=1.
$
\end{propo}

\begin{proof}
The event $\A$ is not local but has the following approximation by local events : for all $\tomega$, we have 
$\1_{\A}(\tomega)= \underset{k \to \infty} {\lim} \1_{\A}(\tomega_{\Lambda_k})$. The local convergence can be used.
\begin{align}\label{eq.calcul.propo.proba.A}
\bar{\nu}(\A)
&=
\underset{k \to \infty}{\lim} \underset{n \to \infty} {\lim}
\int_{\Oc} \1_{\A} ( \tomega_{\Lambda_k}) \bar{\nu}_n(d \tomega) \nonumber
\\ &=
\underset{k \to \infty}{\lim} \underset{n \to \infty} {\lim}
\frac{1}{(2n)^d} \underset{i \in I_n}{\sum}
\int_{\Oc} \1_{\A} ( \tau_i(\tomega_{\Lambda_k})) \tilde{\nu}_n(d \tomega).
\end{align}
For $n >k$, the configuration $\tau_i(\tomega_{\Lambda_k})$ is $\tilde{\nu}_n$-almost surely in $\A$ as soon as $i \in [k-n,n-k]^d$. Therefore
\begin{align}
\bar{\nu}(\A) 
\geq 
\underset{k \to \infty}{\lim} \underset{n \to \infty} {\lim}
\frac{(2(n-k))^d}{(2n)^d} =1. \nonumber
\end{align}
\end{proof}

\begin{propo}\label{propo.WR.unicité.cci}
$$
\bar{\nu}( N_{cc}^{\infty} \leq 1 ) =1,
$$
where the event $\{ N_{cc}^{\infty} \leq 1 \}$ is defined as in Section \ref{section.unicité.ccinfini}.
\end{propo}

\begin{proof}
The proof is based on the results of Section \ref{section.unicité.ccinfini} on the uniqueness of the infinite connected component for the CRCM. First note that the colour-blind probability measure of $\nu_n^{z,Q,q}$, denoted  by $\nu_{n,cb}^{z,Q,q}$, is the finite volume continuum random cluster measure $P_n^{z/q,Q,q}$ defined in Section \ref{section.theo1.bon.candidat}. Therefore, in passing to the limit for a suitable subsequence we find that $
\bar{P}^{z/q,Q,q} = \bar{\nu}^{z,Q,q}_{cb}$. By \ref{propo.unicité} $\bar{P}^{z/q,Q,q}$ produces at most one infinite connected component and by the identification above the same occurs for $\bar{\nu}^{z,Q,q}_{cb}$.\end{proof}

Note that since $\int_{\R^+} R^d Q(dR)= + \infty$, \ref{bigtheo} can not be apply and we don't know if $\bar{P}$ is a CRCM.

\subsubsection{Shield Events}
\label{section.WR.shield}
In this section a sequence $(W_k)$ of "shield" events is introduced in order to localize the indicator function $\1_{\A}(\tomega'_{\Lambda} + \tomega_{\Lambda^c})$.
We define the event 
$$
Col
=
\{
\tomega \in \Oc, \ \tomega \text{ has at least two balls with different colours}
\}.
$$
\begin{propo} \label{propoexistenceshield}
Let $\Lambda$ be a bounded subset of $\R^d$ and let $\nu$ be a stationary probability measure on ($\Oc, \Fc$) satisfying $\nu(Col)=1$. 
Then there exists a sequence $(\Delta_k)_{k \geq 1}$ of compact subset of $\R^d$, a sequence of events $(W_k)_{k \geq 1}$,
satisfying $W_k \in \Fc _{\Delta_k}$ (in particular they are local) such that
\begin{enumerate}
\item $\nu(W_k) \underset{k \to \infty}{\longrightarrow} 1$,

\item for all configurations $\tomega$ in $\A \cap W_k$ and $\tomega'$ in $\Oc$, we have
$$
\1_{\A}( \tomega'_{\Lambda} + \tomega_{\Lambda^c} )
=
\1_{\A}( \tomega'_{\Lambda} + \tomega_{\Delta_k \setminus \Lambda} ).
$$
\end{enumerate}

\end{propo}

\begin{proof} ${}$
Let us begin with the construction of the set $\Delta_k$ and the event $W_k$. The idea is simple. If balls with different colours are wisely placed around $\Lambda$, then those balls prevent those far away to hit the balls in $\Lambda$. Let us give the details.

First since $\Lambda$ is bounded, it is included in a cube 
$\bar{\Lambda}=[-\alpha, \alpha]^d$ for some positive integer $\alpha$. Now for each integer $k$ larger than $\alpha$, we place in each corner of  $\bar{\Lambda}$ a cube $B_j^k$, $j \in \{ 0,1 \}^d$, with edge length $k$;
$$
B_j^k
=
 \underset{i=1..d}{\prod} (-1)^{j_i} [\alpha,k + \alpha],
$$
where $(-1)^{j_i } [\alpha,k + \alpha]= [\alpha, k + \alpha]$ if $j_i= 0$ and $[-\alpha - k ,-\alpha]$ if $j_i=1$. We denote by $W_k^1$ the event of configurations  $\tomega$ having at least two balls with different colours centred inside each $B_j^k$. By simple geometrical arguments there exists a positive integer $D_1$ depending on $k$ such that any ball centred in $\Lambda$, hitting the set $G:=[-\alpha - k -D_1, \alpha  +k +D_1]^d$, necessary covers at least one cube $B_j^k$ for some $j \in \{ 0,1 \}^d$.  So the event $W_k^1$ confines the balls centred in $\Lambda$ inside the set $G$. 
Now we need to prevent balls centred too far away to hit the set $G$. We consider the following $2^d$ cubes
$$
C_j^k
=
 \underset{i=1..d}{\prod} (-1)^{j_i} [\alpha +k +D_1+ 1,\alpha + 2k +D_1+ 1]
$$

for $j \in \{0,1 \} ^d$ and we denote by $W_k^2$ the event of configurations $\tilde \omega$ having at least two balls with different colours centred inside each $C_j^k$. By simple geometrical arguments again, there is a positive integer $D_2$ depending on $k$ such that any ball centred outside 

$$
\Delta_k=[-\alpha - 2k -D_1 -1-D_2, \alpha +2k +D_1 +1 +D_2]^d
$$

hitting the set $G$ necessary covers at least one cube $C_j^k$ for some $j \in \{ 0,1 \}^d$.  So the event $W_k^2$ confines the balls centred in $\Delta_k^c$ inside the set $G^c$. 

\begin{figure}[t]
\begin{center}
   \includegraphics[scale=0.3]{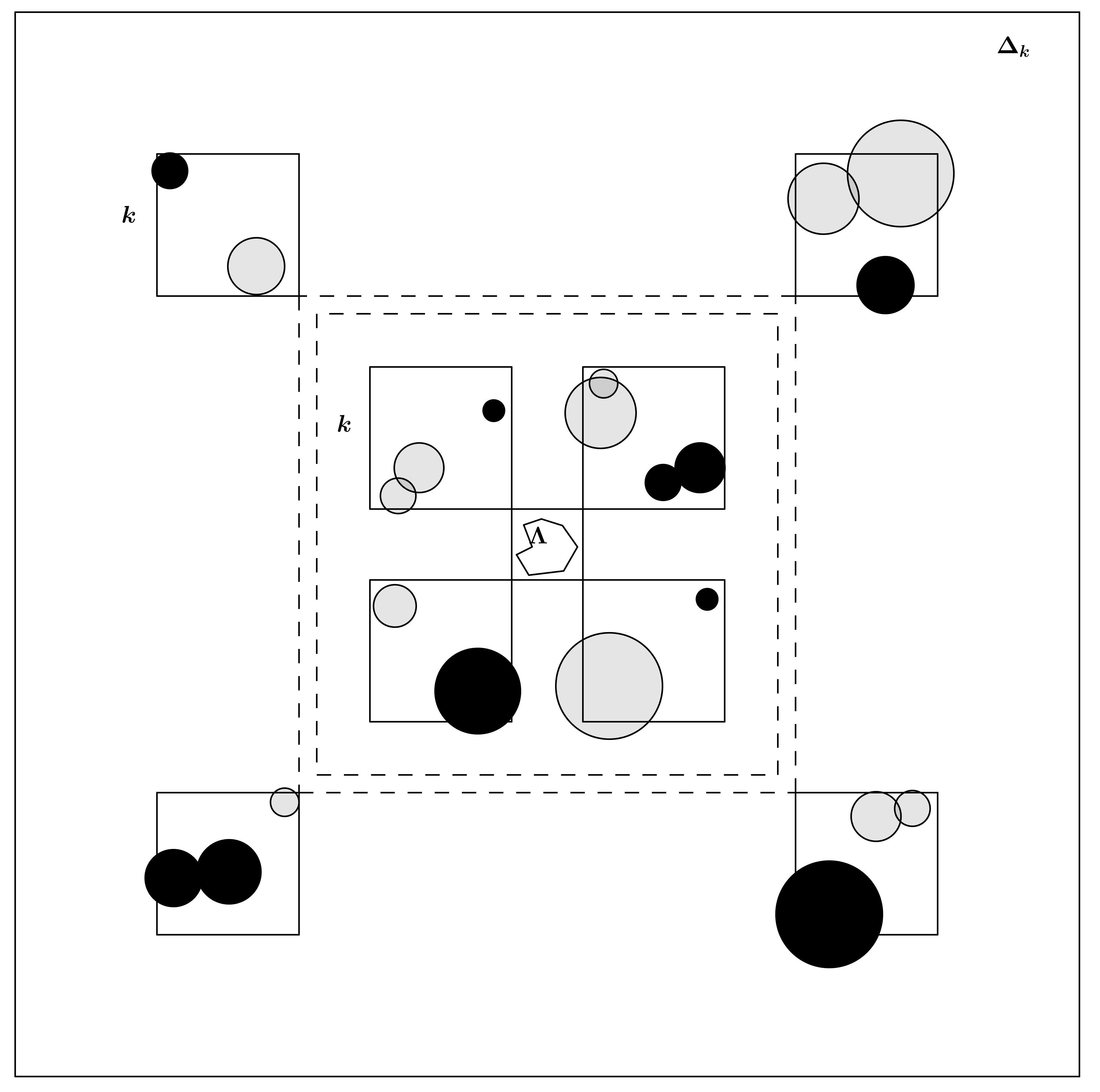}
\end{center}
\caption{Shield event $W_k$}
\end{figure}

In conclusion the event $W_k:=W_k^1\cap W_k^2$ insures that for an allowed configuration $\tilde\omega\in \A$ the balls centred in $\Lambda$ do not hit the balls centred outside $\Delta_k$. The property 2) in \ref{propoexistenceshield} follows. It remains to prove the property 1).

\begin{align}
\nu(W_k^c)
&=
\nu\left( \underset{j \in \{0,1\}^d }{\bigcup}  \{ \tomega_{B_j^k} \not \in Col  \} \cup \{ \tomega_{C_j^k} \not \in Col  \}
  \right) \nonumber
 \\ & \leq
\underset{j \in \{0,1\}^d }{\sum}
\nu( \{ \tomega_{B_j^k} \not \in Col  \}) + \nu( \{ \tomega_{C_j^k} \not \in Col  \}). \nonumber
\end{align}
Since $\nu$ is stationary the probabilities
$\nu( \{ \tomega_{B_j^k} \not \in Col  \})$ and $\nu( \{ \tomega_{C_j^k} \not \in Col  \})$ are equal for all $j$ and fixed $k$. Moreover since $\nu(Col)=1$ this probability tends to $0$ when $k$ tends to infinity. The property 1) follows.

\end{proof}

\subsubsection{The limit point is not monochromatic}
\label{section.WR.preuve.propo.2couleurs}

Ideally we wanted to prove that $\bar{\nu}(Col)=1$ in order to use the shield events $(W_k)$ in \ref{propoexistenceshield}. Unfortunately we didn't success to prove it. However we show that $\bar{\nu}(Col)>0$ for $z$ small enough. It will be enough in the following by considering the probability measure $\bar \nu(.|Col)$ as the expected Widom-Rowlinson model. A probability measure $\mu$ which satisfies $\mu(Col)=0$ is called monochromatic.

\begin{propo}\label{propo.2couleurs}
There is $z_0$ such that, for all $z<z_0$
$$
\bar{\nu}^{z,Q,q}(Col)>0.
$$
\end{propo}

\begin{proof}

We show that $\bar{\nu}$ is different from all stationary monochromatic probability measures by comparing their specific entropy. First we show a uniform lower bound for any  stationary monochromatic probability measures. Secondly we find an upper bound for the specific entropy of $\bar{\nu}$ which is smaller than the lower bound above. These bounds are detailed in the points 1) and 2) below.

{\it 1) The uniform lower Bound}

Let $P$ be a stationary monochromatic probability measure and we suppose first that its colour is deterministic (let us call it red). For any positive integer $n$, we are looking for a lower bound of $\mathcal{I}(P_{\Lambda_n} | \tpi^{z,Q,q}_{\Lambda_n})$. If $P_{\Lambda_n}$ is not absolutely continuous with respect to $\tpi^{z,Q,q}_{\Lambda_n}$ then 
$\mathcal{I}(P_{\Lambda_n} | \tpi^{z,Q,q}_{\Lambda_n})= + \infty$ else $$
\mathcal{I}(P_{\Lambda_n} | \tpi^{z,Q,q}_{\Lambda_n})
=
\int_{\Oc} \ln \left(   \frac{dP_{\Lambda_n}}{d\tpi^{z,Q,q}} (\tomega)\right) \ dP_{\Lambda} (\tomega).
$$

For each configuration $\tomega$ we consider $\tomega^R$ the red balls of $\tomega$ and $\tomega^{NR}$ the others. Then
$$
\frac{dP_{\Lambda_n}}{d\tpi^{z,Q,q}} (\tomega)
=
f_1(\tomega^{NR} |\tomega^{R} )  f_2(\tomega^R),
$$
where 
$f_1( . |\tomega^{R} )$ is the conditional density, with respect to $\tpi^{\frac{q-1}{q}z,Q,q-1}_{\Lambda_n}$, of the non red configurations given $\tomega^R$ 
and $f_2$ the density of red configurations with respect to 
$\tpi_{\Lambda_n}^{z/q,Q,1}$. Since $P$ is monochromatic 
$$
f_1(\tomega^{NR} |\tomega^{R} )
=
\exp\left( \frac{q-1}{q} z |\Lambda_n|  \right) \1_{\emptyset}(\tomega^{NR}), 
$$
and we find
\begin{align*}
\mathcal{I}(P_{\Lambda_n} | \tpi^{z,Q,q}_{\Lambda_n})
&=
\int_{\Oc} \ln( f_1(\tomega^{NR} |\tomega^{R} ) ) P_{\Lambda_n}(d\tomega)+ \int_{\Oc} \ln(  f_2(\tomega^{R}  ) ) P_{\Lambda_n}(d\tomega).
\\ & \geq  \frac{q-1}{q} z |\Lambda_n| .
\end{align*}
Dividing by $|\Lambda_n|$ and taking the limit we obtain this uniform lower bound
\begin{align}
\mathcal{I}^z(P) \geq \frac{q-1}{q}z. \label{eq.minoration.entropie.spe}
\end{align}

Now if the colour of $P$ is not deterministic, then $P$ is a mixture of $q$ monochromatic probability measures with deterministic colour. Since the specific entropy is an affine functional (Proposition 15.14 \cite{g}) the inequality \eqref{eq.minoration.entropie.spe} is still valid for such $P$.

{\it 2) Upper bound for the specific entropy of $\bar{\nu}$}

In \ref{propo.entro.nu n} a first bound $\mathcal{I}^z(\bar{\nu}) \leq z$ is given  but it is not fine enough here since it is larger than the lower bound in \eqref{eq.minoration.entropie.spe}. Let us improve this bound. Recall that
$$
\mathcal{I}^z(\bar{\nu}_n)
=
\frac{1}{|\Lambda_n|} \mathcal{I}(\nu_n | \tpi^{z,Q,q}_{\Lambda_n})
=
\frac{- \ln( \tpi^{z,Q,q}_{\Lambda_n}(\A) )}{|\Lambda_n|}.
$$

Let $y>0$ and $\Delta$ be the cube $]0,y]^d$. $\Lambda_n$ is divided into $k_n$ disjoint copies of $\Delta$ and a boundary term. So $|\Lambda_n| = (2n)^d = k_n |\Delta | + c_n = y^d k_n + c_n $ where $c_n=o(n^d)$. 
We denote by $\phi_y=\frac{1}{y^d} \int_{\R^d} \int_{\R^+} \1_{B(x,R) \subseteq \Delta} Q(dR) dx $ the probability that a ball centred in $\Delta$ is completely included inside $\Delta$. A particular allowed configuration $\tilde \omega\in \A$ can be constructed by forcing that all the balls centred in each copy of $\Delta$ have the same colours and are completely included in $\Delta$. It leads to the following inequality.

\begin{align}\label{eq.minoprobaA}
\tpi^{z,Q,q}_{\Lambda_n}(\A)
&\geq
\left[  \exp(-zy^d) 
\left(  1+ q\underset{i \in \N ^*}{\sum} \frac{1}{i! q^i} (z y^d \phi_y)^i   \right)  \right] ^{k_n} \exp(-z((2n)^d -k_n y^d) ) \nonumber
\\ & \geq
\exp(-z (2n)^d) \times \left(  1+ \underset{i \in \N ^*}{\sum} \frac{q}{i! q^i} (z y^d \phi_y)^i   \right)^{k_n},
\end{align}
and therefore
\begin{align}\label{eq.majoentro}
\mathcal{I}^z(\nu_n^{z,Q,q})
& \leq
z - \frac{k_n}{(2n)^d} \ln \left(  1-q+ q \exp\left( \frac{z y^d \phi_y}{q}\right)  \right) \nonumber
\\ & \leq
z + \frac{1}{y^d}\left( \frac{c_n}{(2n)^d} - 1 \right) 
\ln \left(  1-q+ q \exp\left( \frac{z y^d \phi_y}{q}\right)  \right).
\end{align}

Now since $\frac{c_n}{(2n)^d} \underset{n \to \infty}{\longrightarrow} 0$, there is $n_0$ such that for any $n\geq n_0$ we have $\frac{c_n}{(2n)^d} - 1 \leq - \frac78$.
Proving that the upper bound \eqref{eq.majoentro} is smaller than the lower bounds in \eqref{eq.minoration.entropie.spe}, falls to show that the following function $\Psi$ is negative
 $$
 \Psi : z \mapsto \frac{z}{q} - \frac{7}{8 y^d}
  \ln \left(1-q + q \exp\left( \frac{z y^d \phi_y}{q}\right)  \right).
 $$ 
The derivative is given by
\begin{align}
\Psi'(z)
=
\frac{1}{q} - \frac78 \phi_y \nonumber
\frac{\exp (z y^d \phi_y/q)}{1-q+q \exp(zy^d\phi_y /q )}
\end{align}
which is null for the unique root
$z_y=\frac{q}{\phi_y y^d}  \ln \left(  \frac{q-1}{q} \frac{1}{ 1 -7 \phi_y / 8 }  \right)$. 
This root is positive as soon as $\phi_y > \frac{8}{7q}$, which is realized when $y$ is large enough, since $\phi_y \underset{y \to \infty}{\to} 1$.
With this settings we have $\Psi'(0)<0$ and since $\Psi(0)=0$, the function $\Psi$ is negative at least for $z$ smaller than $z_y$. Hence there exists $z_0>0$ such that for any $0<z<z_c$ there exits $\epsilon>0$ such that for $n\geq n_0$
\begin{align}\label{eq.majoentro2}
\mathcal{I}^z(\nu_n^{z,Q,q})
\leq
\frac{q-1}{q}z-\epsilon.
\end{align}

Since the specific entropy is lower semi-continuous, the inequality \eqref{eq.majoentro2} holds for $\bar{\nu}^{z,Q,q}$ as well. Thanks to the uniform lower bound \eqref{eq.minoration.entropie.spe} for monochromatic probability measures the \ref{propo.2couleurs} is proved.

\end{proof} 
 
\subsubsection{The DLR equations} 
\label{section.WR.DLR}

As mentioned in Section \ref{section.WR.preuve.propo.2couleurs} we  consider the conditional probability measure $\bar{\nu}_{Col} = \bar{\nu}_{Col}^{z,Q,q}= \bar{\nu}(. |Col)$, which is well-defined by \ref{propo.2couleurs}, as the expected Widom-Rowlinson model. Since the event $Col$ is stationary, the probability measure $\bar{\nu}_{Col}$ is still stationary and satisfies $\bar{\nu}_{Col}(\A)=1$ as well. Let us show that $\bar{\nu}_{Col}$ satisfies the DLR equations.

As in \eqref{defimodif} the sequence $(\bar{\nu}_n)$ has to be modified. Let $\Lambda$ be a bounded subset of $\R ^d$. we define
$$
\chi _n^{\Lambda} 
= 
\frac{1}{(2n)^d} \sum\limits_{\substack{i \in I_n \\ \Lambda \subseteq \tau_i(\Lambda_n)}} \nu_n \circ \tau_i^{-1}.
$$

and the analogous of \ref{propo.convlocal.DLR} holds.

Let $f$ be a local measurable function bounded by $1$. We define $\delta$ as followed.
$$
\delta
=
\Big|
\int_{\Oc} f d\bar{\nu}_{Col}
-
\int_{\Oc} \int_{\Oc}
f( \tomega'_{\Lambda} + \tomega_{\Lambda^c} )
\frac{\1_{\A}  ( \tomega'_{\Lambda} + \tomega_{\Lambda^c} )}
{\tilde{Z}_{\Lambda}(\tomega_{\Lambda^c})}
\tpi_{\Lambda}(d\tomega')
\bar{\nu}_{Col}(d\tomega)
\Big|.
$$
Let $\epsilon>0$. By \ref{propoexistenceshield} there exists  $k$ satisfying $\bar{\nu}_{Col}(W^c_k) \leq \epsilon/2$ leading to

\begin{align*}
\delta
& \leq
\Big|
\int_{\Oc} \1_{W_k} f d\bar{\nu}_{Col}
-
\int_{\Oc} \int_{\Oc}
\1_{W_k}(\tomega_{\Lambda^c})
f( \tomega'_{\Lambda} + \tomega_{\Lambda^c} )
\frac{\1_{\A}  ( \tomega'_{\Lambda} + \tomega_{\Lambda^c} )}
{\tilde{Z}_{\Lambda}(\tomega_{\Lambda^c})}
\tpi_{\Lambda}(d\tomega')
\bar{\nu}_{Col}(d\tomega)
\Big| 
+ \epsilon.
\end{align*}
Since $W_k \subseteq Col$ and since $\1_{\A}$ and $\tilde{Z}_{\Lambda}$ are $\Delta_k$-local on $W_k$ by \ref{propoexistenceshield} we obtain
\begin{align*}
\delta & \leq
\Big|
\int_{\Oc} \frac{\1_{W_k}}{\bar{\nu}(Col)} f d\bar{\nu}
-
\int_{\Oc} \int_{\Oc}
\1_{W_k}(\tomega_{\Lambda^c})
f( \tomega'_{\Lambda} + \tomega_{\Lambda^c} )
\frac{\1_{\A}  ( \tomega'_{\Lambda} + \tomega_{\Delta_k \setminus \Lambda} )}
{\tilde{Z}_{\Lambda}(\tomega_{\Delta_k \setminus \Lambda}) \bar{\nu}(Col) }
\tpi_{\Lambda}(d\tomega')
\bar{\nu}(d\tomega)
\Big| 
+ \epsilon.
\end{align*}
The local convergence  $\chi^\Lambda_n\to\bar \nu$ implies that  for $n$ large enough  
\begin{align*}
\delta & \leq
\Big|
\int_{\Oc} \frac{\1_{W_k}}{\bar{\nu}(Col)} f d\chi_n^{\Lambda}
-
\int_{\Oc} \int_{\Oc}
\1_{W_k}(\tomega_{\Lambda^c})
f( \tomega'_{\Lambda} + \tomega_{\Lambda^c} )
\frac{\1_{\A}  ( \tomega'_{\Lambda} + \tomega_{\Delta_k \setminus \Lambda} )}
{\tilde{Z}_{\Lambda}(\tomega_{\Delta_k \setminus \Lambda}) \bar{\nu}(Col) }
\tpi_{\Lambda}(d\tomega')
\chi_n^{\Lambda}(d\tomega)
\Big| 
+ 2\epsilon
\\ & =
\Big|
\int_{\Oc} \frac{\1_{W_k}}{\bar{\nu}(Col)} f d\chi_n^{\Lambda}
-
\int_{\Oc} \int_{\Oc}
\1_{W_k}(\tomega_{\Lambda^c})
f( \tomega'_{\Lambda} + \tomega_{\Lambda^c} )
\frac{\1_{\A}  ( \tomega'_{\Lambda} + \tomega_{ \Lambda ^c} )}
{\tilde{Z}_{\Lambda}(\tomega_{\Lambda ^c}) \bar{\nu}(Col) }
\tpi_{\Lambda}(d\tomega')
\chi_n^{\Lambda}(d\tomega)
\Big| 
+ 2\epsilon
\\ &=
2 \epsilon ,
\end{align*}

where the last equality is due to the DLR equations satisfied by $\chi_n^{\Lambda}$ as in \ref{propo.convlocal.DLR}. Taking $\epsilon$ as small as we want, we get $\delta=0$ and the result.

\section{Heuristic arguments for the conjecture} \label{section.conjecture}

The aim of this Section is to present heuristic arguments which strengthen the conjecture \ref{conjecture}. For simplicity we assume in all the section that the radii are uniformly bounded from below; there exists $R_0>0$, $Q([R_0;+\infty[=1$). Let us start by presenting some rigorous results for any $CRCM(z,q,Q)$ that we denote by $P$.  We don't assume for the moment that the radii are not integrable. 

Since the radii are uniformly bounded from below, the volume of the connected components are also uniformly bounded from below by $v_d R_0^d$. Therefore for any configuration $\omega$ with radii larger than $R_0$, any marked point $X=(x,R)$ and any $\Lambda\subset R^d$ containing $x$

\begin{equation}\label{borneinfenergielocale}
N_{cc}^\Lambda (\omega+\delta_X)-N_{cc}^\Lambda (\omega)\ge -\frac{v_d(R+2R_0)^d}{v_d R_0^d}\ge -C_0 R^d,
\end{equation}
where $C_0$ is the constant $(3/R_0)^d$. We deduce from Theorem 1.1 in \cite{gk} that $P$ dominates stochastically the Poisson process $\pi_{z,\tilde Q}$ where 

$$ \tilde Q(dR)= q^{-C_0R^d}Q(dR).$$ 

Let us note that the measure $ \tilde Q$ is no longer a probability measure (which does not matter), but always admits a $d$-moment; $\int R^d\tilde Q(dR)<+\infty$. This stochastic domination  provides the general behaviour of the connected components of $P$. Indeed it is well known that the germ-grain structure $L(\omega)=\cup_{(x,R)\in\omega} B(x,R)$, under $\pi_{z,\tilde Q}$, percolates for $z$ large enough. Moreover for $z$ very large, $L(\omega)$ is a large ocean of connected balls with a few holes scattered in the space which possibly contain small connected components inside. Since $P$ dominates $\pi_{z,\tilde Q}$, the same behaviour holds for $P$ or $P$ produces a large ocean of connected balls without holes. In this second case, $P$ has only one infinite connected component. The conjecture claims that, in the non integrable setting, for $z$ large enough this second behaviour occurs. 

Let us define the quantity $N_P$ which represents the mean number of connected components per unit volume produced by $P$. Let $X$ be an element of $\omega$ and $C_X(\omega)$ the connected components of $L(\omega)$ containing $X$. We say that $X=(x,R)$ is the far left point in $C_X(\omega)$ if the first coordinate of $x$ in $\R^d$ is smaller than any first coordinate of point $Y\in C_X(\omega)$. For $P$ almost all $\omega$, any bounded connected component in $L(\omega)$ has only one far left point. So there is a bijection between the connected components and the far left points. Therefore a possible definition of $N_P$ is

$$ N_P=\int_{\Omega} \sum_{(x,R)\in\omega} \1_{[0,1]^d}(x) \1_{\{(x,R) \text{ is the far left point in } C_X(\omega)\}} P(d\omega).$$

The behaviour of the connected components of $P$ described above implies that $N_P$ goes to zero when $z$ goes to infinity. We can show rigorously an exponential decay.

\begin{lemme} \label{lemmedecroissance} Assume that the radii are uniformly bounded from above; there exists $R_0>0$ such that $Q([R_0, +\infty[)=1$. Then there exist $C>0$ such that, for $z$ large enough,
$$ N_P\le e^{-Cz}.$$

\end{lemme}

\begin{proof}

Thanks to the GNZ equation \eqref{eq.GNZ.CRCM} and the stationarity of $P$, we have

\begin{eqnarray*}
N_P & =& z \int_{\Omega} \int_{[0,1]^d} \int_0^{+\infty}  q^{N^\Lambda_{cc}(\omega+\delta(x,R))-N^\Lambda_{cc}(\omega)}\\
& & \1_{\{(x,R) \text{ is the far left point in } C_{(x,R)}(\omega+\delta_{(x,R)})\}} Q(dR) dx P(d\omega)\\
& \le & zq   \int_0^{+\infty} P\Big((0,R) \text{ is the far left point in } C_{(0,R)}(\omega+\delta_{(0,R)})\Big) Q(dR)  
\end{eqnarray*}

Since that $P$ dominates $\pi_{z,\tilde Q}$ we get

\begin{eqnarray}\label{decroissance}
N_P & \le & zq  \int_0^{+\infty} \pi_{z,\tilde Q}\Big((0,R) \text{ is the far left point in } C_{(0,R)}(\omega+\delta_{(0,R)})\Big) Q(dR) 
\nonumber \\ 
& \le & zq   \int_0^{+\infty} \pi_{z,\tilde Q}\Big( 0\notin L(\omega_{\text{left}})\Big) Q(dR) 
\nonumber \\
& = & zq   e^{-z\frac 12 v_d \int_0^{+\infty} R^d \tilde Q(dR)},
\end{eqnarray}
 
where $\omega_{\text{left}}$ is the configuration of points $(x,R) \in\omega$ whom the first coordinate of $x$ is negative. The last equality in \eqref{decroissance} comes from standard computations for the Boolean model \cite{chiu2013}. In adjusting correctly the constant $C$ the proof of the lemma follows.

\end{proof}

Let us turn now to a non rigorous proof of the conjecture in the case  $Q(dR)=\frac{d-1}{R^d}\1_{[1,+\infty[}(R)dR$. Other distribution $Q$ could have been considered as well. Let us show that $N_P=0$ for $z$ large enough which leads to $P=\pi_{z,Q}$. 

We assume that $P$ is extremal (so ergodic) in the simplex of $CRCM(z,q,Q)$. Since $N_P$ is the mean number of connected components per unit volume, thanks to the ergodic Theorem, for any $x\in\R^d$ and $P$-almost every $\omega$ 

$$ \lim_{R\mapsto \infty}   \frac{N_{cc}^\Lambda (\omega+\delta_{(x,R)})-N_{cc}^\Lambda (\omega)}{v_d R^d} = -N_P.$$

So there exists $K_{x,\omega}>0$ such that for all $R>0$

$$
N_{cc}^\Lambda (\omega+\delta_{(x,R)})-N_{cc}^\Lambda (\omega)
\ge -2 v_d N_P R^d-K_{x,\omega}.$$

Assume that this constant $K$ can be chosen uniformly in $x$ and $\omega$. Obviously this is wrong but, choosing $K$ very large, this inequality holds with high probability which gives a sense to this approximation. It is {\bf the unique non rigorous part of this section}. 

Following the same computations as in the proof of \ref{lemmedecroissance} we obtain that

\begin{eqnarray*}\label{equationimpossible}
N_P & \le &  zq   e^{-\frac 12 z v_d \int_0^{+\infty} R^d q^{-2 v_d N_P R^d -K} Q(dR)} \\
& = &  zq   e^{-\frac 12 z v_d q^{-K}\int_1^{+\infty} q^{-2 v_d N_P R^d} dR} \\
& \le &  zq   e^{-cz N_P^{-1/d} }  
\end{eqnarray*}
where $c$ is a non negative constant. A simple analysis of this inequality shows that for $z$ large enough, the only one solution is $N_P=0$.\\

{\it Acknowledgement:} This work was supported in part by the Labex CEMPI  (ANR-11-LABX-0007-01).

\nocite{*}
\bibliographystyle{plain}
\bibliography{biblio}

\begin{thebibliography}{10}

\bibitem{bc}
C.~Borgs and J.~T. Chayes.
\newblock The covariance matrix of the {P}otts model: a random cluster
  analysis.
\newblock {\em J. Statist. Phys.}, 82(5-6):1235--1297, 1996.

\bibitem{burkea}
R.~M. Burton and M.~Keane.
\newblock Density and uniqueness in percolation.
\newblock {\em Comm. Math. Phys.}, 121(3):501--505, 1989.

\bibitem{cck}
J.~T. Chayes, L.~Chayes, and R.~Koteck{\'y}.
\newblock The analysis of the widom-rowlinson model by stochastic geometric
  methods.
\newblock {\em Comm. Math. Phys.}, 172(3):551--569, 1995.

\bibitem{chiu2013}
Sung~Nok Chiu, Dietrich Stoyan, Wilfrid~S Kendall, and Joseph Mecke.
\newblock {\em Stochastic geometry and its applications}.
\newblock John Wiley \& Sons, 3 edition, 2013.

\bibitem{david}
David Dereudre.
\newblock The existence of quermass-interaction processes for nonlocally stable
  interaction and nonbounded convex grains.
\newblock {\em Adv. in Appl. Probab.}, 41(3):664--681, 2009.

\bibitem{dereudredrouilhetgeorgii}
David Dereudre, Remy Drouilhet, and Hans-Otto Georgii.
\newblock Existence of {G}ibbsian point processes with geometry-dependent
  interactions.
\newblock {\em Probab. Theory Related Fields}, 153(3-4):643--670, 2012.

\bibitem{GZ}
H.-O. Georgi and H~Zessin.
\newblock Large deviations and the maximum entropy principle for marked point
  random fields.
\newblock {\em Probability Theory and Related Fields}, 96:177--204, 1993.

\bibitem{gh}
H.-O. Georgii and O.~H{\"a}ggstr{\"o}m.
\newblock Phase transition in continuum {P}otts models.
\newblock {\em Comm. Math. Phys.}, 181(2):507--528, 1996.

\bibitem{g}
Hans-Otto Georgii.
\newblock {\em Gibbs measures and phase transitions}, volume~9 of {\em de
  Gruyter Studies in Mathematics}.
\newblock Walter de Gruyter \& Co., Berlin, second edition, 2011.

\bibitem{ghm}
Hans-Otto Georgii, Olle H{\"a}ggstr{\"o}m, and Christian Maes.
\newblock The random geometry of equilibrium phases.
\newblock In {\em Phase transitions and critical phenomena, {V}ol. 18},
  volume~18 of {\em Phase Transit. Crit. Phenom.}, pages 1--142. Academic
  Press, San Diego, CA, 2001.

\bibitem{gk}
Hans-Otto Georgii and Torsten K{\"u}neth.
\newblock Stochastic comparison of point random fields.
\newblock {\em J. Appl. Probab.}, 34(4):868--881, 1997.

\bibitem{grim}
Geoffrey Grimmett.
\newblock The stochastic random-cluster process and the uniqueness of
  random-cluster measures.
\newblock {\em Ann. Probab.}, 23(4):1461--1510, 1995.

\bibitem{grimbook}
Geoffrey Grimmett.
\newblock {\em The random-cluster model}, volume 333 of {\em Grundlehren der
  Mathematischen Wissenschaften [Fundamental Principles of Mathematical
  Sciences]}.
\newblock Springer-Verlag, Berlin, 2006.

\bibitem{meeroy}
Ronald Meester and Rahul Roy.
\newblock {\em Continuum percolation}, volume 119 of {\em Cambridge Tracts in
  Mathematics}.
\newblock Cambridge University Press, Cambridge, 1996.

\bibitem{A-Moller08}
J.~M\o{}ller and K.~Helisov\'a.
\newblock Power diagrams and interaction processes for union of discs.
\newblock {\em Adv. Appli. Prob.}, 40(2):321--347, 2008.

\bibitem{A-MollerHel10}
J.~M\o{}ller and K.~Helisov\'a.
\newblock Likelihood inference for unions of interacting discs.
\newblock {\em Scandinavian Journal of Statistics}, 37(3):365--381, 2010.

\bibitem{NZ}
X.~Nguyen and H.~Zessin.
\newblock {Integral and differential characterizations {G}ibbs processes}.
\newblock {\em Mathematische Nachrichten}, 88(1):105--115, 1979.

\bibitem{Widom70}
B.~Widom and J.S. Rowlinson.
\newblock New model for the study of liquid-vapor phase transitions.
\newblock {\em J. Chem. Phys.}, 52:1670--1684, 1970.

\end{thebibliography}

\end{document}